\documentclass[11pt]{amsart}
\usepackage{graphicx} 

\usepackage[utf8]{inputenc}
\usepackage[margin=1in,letterpaper,portrait]{geometry}
\usepackage{ytableau}
\ytableausetup{notabloids}
\ytableausetup{centertableaux}
\usepackage{latexsym,amsmath,amssymb,verbatim, tikz}
\usepackage{graphicx}
\usepackage{hyperref}

\newtheorem{theorem}{Theorem}[section]
\newtheorem{proposition}[theorem]{Proposition}

\newtheorem{lemma}[theorem]{Lemma}
\newtheorem{corollary}[theorem]{Corollary}
\newtheorem{conjecture}[theorem]{Conjecture}
\newtheorem{problem}[theorem]{Problem}
\newtheorem{observation}[theorem]{Observation}

\newtheorem{definition}[theorem]{Definition}
\newtheorem{defn}[theorem]{Definition}
\newtheorem{example}[theorem]{Example}

\newtheorem{remark}[theorem]{Remark}

\numberwithin{equation}{section}

\newcommand{\todo}[1]{\vspace{5 mm}\par \noindent
	\marginpar{\textsc{ToDo}} \framebox{\begin{minipage}[c]{0.95
				\textwidth}
			#1 \end{minipage}}\vspace{5 mm}\par}

\newcommand{\Q}{{\mathcal {Q}}}

\newcommand{\CC}{{\mathbb {C}}}

\newcommand{\RR}{{\mathbb {R}}}
\newcommand{\ZZ}{{\mathbb {Z}}}

\newcommand{\spn}{{\operatorname{span}}}

\newcommand{\ve}{\varepsilon}

\newcommand{\rank}{{\operatorname{rank}}}

\newcommand{\Des}{{\operatorname{Des}}}

\newcommand{\ch}{{\operatorname{ch}}}

\newcommand{\SYT}{{\operatorname{SYT}}}

\newcommand{\Short}{{\operatorname{Short}}}

\newcommand{\sign}{{\operatorname{sign}}}

\newcommand{\Hilb}{{\operatorname{Hilb}}}

\newcommand{\symm}{{\mathfrak{S}}}

\newcommand{\DDD}{{\mathcal{D}}}
\newcommand{\EEE}{{\mathcal{E}}}
\newcommand{\UUU}{{\mathcal{U}}}

\newcommand{\F}{\mathcal{F}}

\newcommand{\uX}{\underline{X}}
\newcommand{\uY}{\underline{Y}}
\newcommand{\uYr}{\underline{Y_r}}
\newcommand{\uYi}{\underline{Y_i}}

\newcommand{\edgeright}{{\operatorname{minimal}}}

\newlength{\mysizetiny}
\newlength{\mysizesmall}
\newlength{\mysize}
\newlength{\mysizelarge}
\title[Transitive and Gallai colorings]
{Transitive and Gallai colorings}

\author[R. M. Adin]{Ron M.\ Adin}
\address{Department of Mathematics, Bar-Ilan University, Ramat-Gan 52900, Israel}
\email{radin@math.biu.ac.il}
\author[A. Berenstein]{Arkady Berenstein}
\address{Department of Mathematics, University of Oregon, Eugene, OR 97403, USA}
\email{arkadiy@math.uoregon.edu}
\author[J. Greenstein]{Jacob Greenstein}
\address{Department of Mathematics, University of California, Riverside,
	CA 92521, USA}
\email{jacobg@ucr.edu}
\author[J.-R. Li]{Jian-Rong Li}
\address{Faculty of Mathematics, University of Vienna, Oskar-MorgensternPlatz 1, 1090 Vienna, Austria}
\email{lijr07@gmail.com} 
\author[A. Marmor]{Avichai Marmor}
\address{Department of Mathematics, Bar-Ilan University, Ramat-Gan 52900, Israel}
\email{avichai@elmar.co.il}
\author[Y. Roichman]{Yuval Roichman}
\address{Department of Mathematics, Bar-Ilan University, Ramat-Gan 52900, Israel}
\email{yuvalr@math.biu.ac.il}

\date{September 20, 2023}

\thanks{RMA and YR were partially supported by the Israel Science Foundation, grant no.~1970/18. 
	AB was partially supported by Simons foundation collaboration grant no.~636972. 
	JG was partially supported by Simons foundation collaboration grant no.~245735. JRL was partially supported by the Austrian Science Fund (FWF) no.~P-34602. AM was  partially supported by  the European Research Council under the ERC starting grant agreement no.~757731 (LightCrypt) and  by the Israel Science Foundation, grant no.~1970/18. 
}

\begin{document}
	
	\begin{abstract}
		A Gallai coloring of  the complete graph is an edge-coloring with no rainbow triangle.  
		This concept first appeared in the study of comparability graphs and anti-Ramsey theory. 
		We introduce a transitive analogue for acyclic directed graphs, and generalize both notions to Coxeter systems, matroids and commutative algebras.  
		
		It is shown that for any finite matroid (or oriented matroid), 
		the maximal number of colors is equal to the matroid rank. 
		This generalizes a result of  Erd\H{o}s-Simonovits-S\'os  for complete graphs. 
		The number of Gallai (or transitive) colorings of the matroid that use at most $k$ colors  
		is a polynomial in $k$. 
		Also, for any acyclic oriented matroid, represented over the real numbers, the number of transitive colorings 
		using at most 2 colors is equal  to the number of chambers in the dual hyperplane arrangement.  
		
		
		We count Gallai and transitive colorings of the root system of type $A$ using the maximal number of colors, 
		and show that, when equipped with a natural descent set map, the resulting quasisymmetric function is symmetric and Schur-positive.
		
		
	\end{abstract}
	
	\maketitle
	\tableofcontents

	\section{Introduction}

	\subsection{Gallai and transitive colorings of matroids}
	\label{sec:Gallai_enumeration}
	
	
	
	A {\em Gallai coloring} of  the complete graph  
	$K_n$ on $n$ vertices
	is an edge-coloring 
	which has no {\em rainbow triangle}, namely a triangle with edges of (three) different colors. 
	This concept was applied in a seminal paper of Gallai~\cite{Gallai} to characterize comparability graphs.
	It was named after Gallai by Gy\'arf\'as and Simonyi~\cite{GS04}.
	Various extensions of the definition to general graphs were offered; see, e.g.,~\cite{GS10}  
	and~\cite{Gouge_etal}. 
	In this paper we adopt the definition of Gouge et al.~\cite{Gouge_etal}, 
	which appeared implicitly already in~\cite{Haxell}, 
	and extend it to the context of matroids.
	
	For a positive integer $k$ denote $[k]:=\{1,2,\dots,k\}$.
	
	\begin{definition}\label{def:Gallai}
		Let $k$ be a positive integer, and let $M$ be a matroid on a finite set $E$. 
		A {\em Gallai $k$-coloring} of $M$ is a function $\ve: E \to [k]$ such that, for any circuit $X$ in $M$,
		\[
		|\{\ve(e) \,:\, e \in X\}| < |X|.
		\]
		In particular, a Gallai $k$-coloring of the graphic matroid corresponding to a graph $G =(V,E)$ is an edge coloring $\ve: E \to [k]$ with no rainbow cycle. 
	\end{definition}
	
	Gallai colorings were extensively studied; see the survey paper~\cite{FMO} and references therein. 
	For recent results  regarding the combinatorial structure and asymptotic enumeration of maximal Gallai colorings of the complete graph  see~\cite{Balogh, Bastos2, Bastos}. 
	
	
	
	Berenstein, Greenstein and Li~\cite{BGL} introduced, in their study of monomial braidings, the concept of a {\em $(\Gamma,C)$-transitive function} for any directed graph $\Gamma$ and set of colors $C$, see e.g.  Example~\ref{example:tournament} below.
	Motivated by this work, we define transitive colorings of general oriented matroids.
	
	\begin{definition}\label{def:transitive}
		Let $k$ be a positive integer and let $M$ be an oriented matroid on a finite set $E$. 
		A {\em transitive $k$-coloring} of $M$ is a function $\ve: E \to [k]$ 
		such that, for any signed circuit $X = (X^+,X^-)$ in $M$, 
		\[
		\{\ve(e):\ e\in X^+\} \cap \{\ve(e):\ e\in X^-\}\ne \varnothing,
		\]
		In particular, a transitive $k$-coloring of the oriented matroid corresponding to a directed graph $G =(V,E)$ is an edge-coloring $\ve: E \to [k]$ such that any cycle contains two directed edges with the same color but opposite orientations. 
	\end{definition}
	
	
	Observe that the set of transitive $k$-colorings of an oriented matroid may be identified with a proper subset of the set of Gallai $k$-colorings of the underlying (unoriented) matroid. 
	
	
	Here is a reformulation of Definition~\ref{def:transitive}  for {\em representable} oriented matroids.
	
	\begin{definition}\label{def:transitive_rep}
		Let $k$ be a positive integer, and let $E$ be a finite set (or multiset) of vectors in a vector space over an ordered field (say, the field $\RR$ of real numbers).
		A {\em transitive $k$-coloring} of $E$ is a function $\ve: E \to [k]$
		such that, for any two disjoint subsets $S,T \subseteq E$,
		\[
		\spn_{\RR_{> 0}}(S) \cap \spn_{\RR_{> 0}}(T) \ne \varnothing \,\Longrightarrow\, 
		\ve(S) \cap \ve(T) \ne \varnothing. 
		\]
	\end{definition}
	
	\begin{observation}\label{obs:111}
		Every transitive coloring of a 
		set of vectors
		satisfies the following condition: 
		\[
		u \in \spn_{\RR_{> 0}} \{v_1,\ldots,v_t\} \,\Longrightarrow\, 
		\ve(u) \in \{\ve(v_1),\ldots,\ve(v_t)\}.
		\]
		This condition is equivalent to the one in Definition~\ref{def:transitive_rep} in certain important cases 
		(e.g., tournaments and Coxeter root systems), 
		but not in general.
	\end{observation}
	
	In particular, every directed graph may be viewed as an oriented matroid represented over (any) ordered field.
	
	\begin{example}
		Consider the two 3-colorings of acyclic directed graphs depicted in Figure~\ref{fig:On}.   
		\begin{figure}[htb]
			\begin{center}
				\begin{tikzpicture}[scale=0.5]
				\draw[red]			(0,0)--(0,4);
				\draw[blue]			(4,0)--(4,4);
				\draw[red]			(0,4)--(4,4);
				\draw[green]			(4,0)--(0,0);
				
				\draw[black] (2,0) node {$>$}; 
				\draw[black] (2,4) node {$>$}; 				
				\draw[black] (0,2) node {$\wedge$};	
				\draw[black] (4,2) node {$\wedge$}; 
				
				\draw[red]			(8,0)--(8,4);
				\draw[blue]			(12,0)--(12,4);
				\draw[red]			(8,4)--(12,4);
				\draw[green]			(12,0)--(8,0);
				
				\draw[black] (10,0) node {$>$}; 
				\draw[black] (10,4) node {$<$}; 				
				\draw[black] (8,2) node {$\wedge$};	
				\draw[black] (12,2) node {$\wedge$}; 
				
				\end{tikzpicture}
			\end{center}
			\caption{A non-transitive coloring and a transitive coloring}\label{fig:On}
		\end{figure}
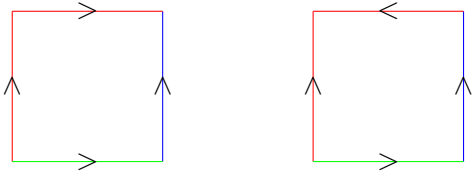
		The coloring on the left satisfies the condition in Observation~\ref{obs:111},
		but is not transitive.
		The coloring on the right is transitive. 
	\end{example}
	
	\begin{example}\label{example:tournament}
		Let $\overrightarrow K_n$ be the transitive tournament with vertex set $\{1,\dots,n\}$ 
		and edge set $\{(i,j)\,:\, i<j\}$. 
		This is an acyclic orientation of the complete graph. 
		An edge-coloring $\ve$ of $\overrightarrow K_n$ is {\em transitive} if and only if 
		\[
		\ve(i,k) \in \{\ve(i,j),\ve(j,k)\} 
		\qquad (\forall\, i<j<k).
		\] 
	\end{example}
	
	

	

	
	The anti-Ramsey problem posed by Erd\H{o}s, Simonovits and S\'os~\cite{ESS}
	asks for the maximal number $k$ of colors such that there exists an edge-coloring of the complete graph of order $n$, $K_n$, with exactly $k$ colors and without a rainbow complete subgraph $K_s$. They proved, in particular,  that the maximal number of edge colors of $K_n$ without a rainbow triangle is $n-1$. 
	We generalize this result to any matroid.
	
	
	\begin{observation}\label{t:must_be_loopless_intro} {\rm (Remark~\ref{rem:must_be_loopless_intro}(a) below)} 
		A matroid has a Gallai coloring if and only if it is loopless (i.e., has no circuit of size $1$). 
	\end{observation}
	
	For a loopless matroid $M$, let $g(M)$ be the maximal $k$ such that there exists a Gallai coloring of $M$ using exactly $k$ colors. 
	
	\begin{theorem}\label{t.Gallai_max_eq_rank_intro}
		{\rm (Theorem~\ref{t.Gallai_max_eq_rank} below)}
		For any loopless matroid $M$,  
		\[
		g(M) = \rank(M).
		\]
	\end{theorem}
	
	The following corollary generalizes the Erd\H{o}s-Sinonovits-S\'os result. 
	
	\begin{corollary}\label{t.Gallai_max_graph_intro}
		{\rm (Corollary~\ref{t.Gallai_max_graph} below)} 
		The maximal number of colors in a Gallai coloring of a graph $G$ on $n$ vertices with $c$ connected components is $n-c$.
	\end{corollary}
	
	A similar result holds for transitive colorings.
	
	
	
	
	\begin{observation}\label{t:must_be_acyclic_intro} 
		An oriented matroid has a transitive coloring if and only if it is acyclic (i.e., has no positive circuit). 
		In particular, a nonempty set of vectors $E$ has a transitive coloring (even with a single color) if and only if
		for every nonempty subset $S \subseteq E$,
		\[
		0 \not\in \spn_{\RR_{> 0}}(S);
		\]
		equivalently, if and only if $0$ is not in the convex hull of $E$.
	\end{observation} 
	
	For an acyclic oriented matroid $M$,
	let $t(M)$ be the maximal $k$ such that there exists a transitive coloring of $M$ using exactly $k$ colors. 
	
	\begin{theorem}\label{t:max_transitive_eq_rank_intro}
		{\rm (Theorem~\ref{t:max_transitive_eq_rank} below)}
		For any acyclic oriented matroid $M$,  
		\[
		t(M) = \rank(M).
		\]
	\end{theorem}
	
	\begin{corollary}\label{t.transitive_max_graph_intro}
		The maximal number of colors in a transitive coloring of the set of positive roots of a Coxeter group $W$ is equal to its rank.
	\end{corollary}
	
	
	
	
	The concept of Gallai partitions of complete graphs was introduced by K\"orner, Simonyi and Tuza~\cite{KST92}. It is naturally generalized to all graphs and matroids.
	
	\begin{definition}\label{def:Gallai_partition} 
		Let $M$ be a loopless matroid on a nonempty set $E$. 
		A {\em Gallai $k$-partition} of $M$ 
		is a partition of $E$ into $k$ disjoint non-empty subsets, also called {\em blocks}, $B_1, \ldots, B_k$, 
		such that for any circuit $X$ in $M$,
		$|X \cap B_i| \ge 2$ for at least one value of $i$.
	\end{definition}
	
	There is a transitive analogue.
	
	\begin{definition}\label{def:transitive_partition}
		Let $M$ be an acyclic oriented matroid on a nonempty set $E$. 
		A {\em transitive $k$-partition} of $M$ 
		is a partition of $E$ into $k$ disjoint non-empty subsets, also called {\em blocks}, $B_1, \ldots, B_k$, 
		such that for any signed circuit $X = (X^+,X^-)$ in $M$, both 
		$X^+ \cap B_i \ne \varnothing$ and $X^- \cap B_i \ne \varnothing$ for at least one value of $i$.
	\end{definition}
	
	
	The following result 
	resolves~\cite[Conjecture 3.5]{BGL} 
	as a special case.
	
	\begin{proposition}
		For any loopless (respectively, acyclic oriented) matroid $M$ on a nonempty set $E$ 
		there exists a polynomial $p_M(x) \in x\ZZ[x]$ 
		such that, for any positive integer $k$,  
		the number of Gallai (respectively, transitive)  colorings of $M$ using 
		$k$ colors is equal to $p_M(k)$.
		Specifically, 
		\[
		p_M(x) = \sum_{j \ge 1} a_j (x)_j,
		\]
		where $(x)_j := x(x-1) \cdots (x-j+1)$ and $a_j$ is the number of Gallai (respectively, transitive) $j$-partitions of $M$. 
	\end{proposition}
	
	See Propositions~\ref{t.polynomial} and~\ref{G.polynomial} below.
	
	\smallskip
	
	While counting Gallai $2$-colorings of a loopless matroid is easy (Proposition~\ref{prop:2_Gallai}), 
	the enumeration of transitive $2$-colorings 
	of an acyclic oriented matroid is more involved.
	
	\begin{theorem}
		{\rm (Theorem~\ref{thm:acyclic_2_colorings} below)}
		Let $M$ be an acyclic oriented matroid on a nonempty set $E$.
		The number of transitive 2-colorings of $M$ is equal to $2^c$ times the number of acyclic reorientations of $M$, 
		where $c$ is the number of connected components of $M$. 
	\end{theorem}
	
	For representable oriented matroids we prove the following. 
	
	\begin{theorem}\label{thm:Orlik} {\rm (Theorem~\ref{thm:Orlik1} below)}	
		For any acyclic oriented matroid $M$ represented over $\RR$, 
		the number of transitive $2$-colorings of $M$ is equal to the number of chambers in the dual hyperplane arrangement. 
	\end{theorem}	
	
	In particular, for any finite Coxeter group $W$, the number of transitive $2$-colorings of the set $\Phi^+(W)$ of positive roots is equal to $|W|$; see Corollary~\ref{cor:2-Coxeter} below.
	Theorem~\ref{thm:Orlik} is closely related to a well-known result of Orlik and Terao~\cite{OT}; see the discussion in Section~\ref{sec:OT}.

	\subsection{Type $A$: enumeration and Schur-positivity}
	\label{sec:Descents_Schur}
	
	
	Gallai and transitive colorings of the root system of type $A_{n-1}$
	\[
	\Phi^+ (A_{n-1})= \{e_i-e_j:\ 1\le i<j\le n\}
	\]
	may be interpreted as edge-colorings of the undirected (respectively, directed) complete graph of order $n$. 
	Asymptotic results about the number of Gallai edge colorings of complete graphs were obtained recently~\cite{Balogh, Bastos2, Bastos}. 
	In particular, it was proved that, 
	for any fixed $k \ge 2$ and sufficiently large $n$, almost all Gallai colorings of $K_n$ using at most $k$ colors actually use only two colors. 
	Some results regarding precise counting were obtained by Gouge et al.~\cite{Gouge_etal}. 
	
	Recall Definitions~\ref{def:Gallai_partition} and~\ref{def:transitive_partition}. 
	A Gallai (transitive) partition is {\em maximal} if the number of blocks is maximal, namely 
	(by Theorems~\ref{t.Gallai_max_eq_rank_intro}  and~\ref{t:max_transitive_eq_rank_intro})
	equal to the rank of the matroid.  
	We prove the following.
	
	\begin{theorem}\label{t.number_max_Gallai_partitions1} {\rm(Theorem~\ref{t.number_max_Gallai_partitions} below)} 
		For every $n>1$, the number of maximal Gallai partitions of the set of edges of the complete graph $K_n$ is equal to the double factorial $(2n-3)!!$.
	\end{theorem}
	
	
	
	
	
	
	
	
	
	
	\begin{theorem}\label{thm:Catalan1} 
		{\rm(Theorem~\ref{thm:Catalan} below)} 
		For every $n>1$, the number of maximal transitive partitions of the set of edges of the transitive tournament ${\overrightarrow K_n}$ is equal to the Catalan number $C_{n-1}:=\frac{1}{n}\binom{2n-2}{n-1}$.
	\end{theorem}
	
	For a $q$-analogue see Proposition~\ref{prop:q-analog}.

	
	\medskip
	
	We further consider quasisymmetric generating functions (i.e., refined counts with respect to a certain set-valued function), and prove that they are symmetric and Schur-positive for any number of colors. 
	
	\smallskip
	
	A symmetric function is called Schur-positive if all the coefficients in its expansion in the Schur basis
	are nonnegative (or polynomials with nonnegative coefficients).
	Deciding the Schur-positivity of a given symmetric function is equivalent, via the characteristic map, to showing
	that a given class function is actually a character, and is a frequently encountered problem in contemporary
	algebraic combinatorics; see, e.g.,~\cite[Ch. 3]{Stanley_problems}.
	
	Recall the 
	fundamental quasisymmetric function indexed by a subset $J\subseteq [n-1]$:
	\[
	\F_J(\bf x) :=
	\sum_{\substack{i_1 \le i_2 \le \ldots \le i_n \\
			{i_j<i_{j+1}\ \text{ if }\ j\in J}}} 
	x_{i_1} x_{i_2} \cdots x_{i_n}.
	\]
	For a set $A$ of combinatorial objects, equipped with a map $\Des: A \to 2^{[n-1]}$, let
	\[
	\Q(A):=\sum\limits_{a\in A} \F_{\Des(a)}.
	\]
	The quasisymmetric function $\Q(A)$ was introduced by Gessel in~\cite{Gessel}. 
	Gessel was motivated by a well-known conjecture of Stanley~\cite[III, Ch. 21]{Stanley_Memoir}, 
	which he reformulates as follows:
	if $A$ is the set of linear extensions of a labeled poset $P$, then $\Q(A)$ is symmetric if and only if $P$ is isomorphic to the
	poset determined by a skew semistandard Young tableau.
	The following problem was posed by Gessel and Reutenauer~\cite{GR93} in the context of permutation sets.
	
	\begin{problem}
		For which pairs $(A,\Des)$ is $\Q(A)$ symmetric and Schur-positive?
	\end{problem}
	
	
	
	\begin{definition} 
		The {\em descent set} of a Gallai (respectively, transitive) $k$-partition $p$ 
		of the complete graph $K_n$ (respectively, the transitive tournament $\overrightarrow K_n$) on the set of vertices $\{1, \ldots, n\}$ is 
		\[
		\Des(p) 
		:= \{ i \,:\, \text{the edge } (i,i+1) \text{ forms a singleton block in } p\}.
		\]
	\end{definition}
	
	\begin{example}
		Figure~\ref{fig:Des1} shows the descent sets of two Gallai partitions of $K_4$, where the edges of distinct blocks are colored by distinct colors. 
		Note that in the paritition on the right, the edge $(1,3)$ forms a singleton block, but is not a descent since its endpoints do not have consecutive labels.
		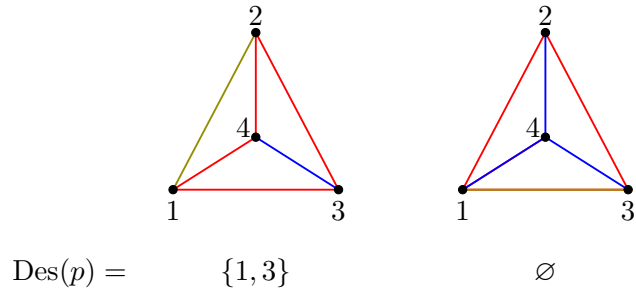
\begin{figure}[htb]
			\begin{center}
				\begin{tikzpicture}[scale=0.55]
				\draw[line width=0.25mm, red]			(2,1.27)--(0,0)--(4,0)--(2,3.8)--(2,1.27);
				\draw[line width=0.25mm, blue] (4,0)--(2,1.27);
				\draw[line width=0.25mm, olive] (0,0)--(2,3.8); 
				
				\draw[line width=0.25mm, red]			(9,1.27)--(7,0)--(11,0)--(9,3.8);
				\draw[line width=0.25mm, blue] (7,0)--(9,1.27)--(11,0);
				\draw [line width=0.25mm, blue](9,1.27)--(9,3.8); 
				\draw[line width=0.25mm, red] (9,3.8)--(7,0); 
				\draw[line width=0.25mm, olive] (7,0)--(11,0);
				
				\draw[fill] (0,0) circle (.1); 	\draw[fill] (4,0) circle (.1); 	\draw[fill] (2,3.8) circle (.1); \draw[fill] (2,1.27) circle (.1);
				\draw (0,-0.5) node {$1$}; 	
				\draw (4,-0.5) node {$3$}; 	
				\draw (2, 4.2) node {$2$}; 
				\draw (1.7, 1.5) node {$4$};
				
				\draw[fill] (7,0) circle (.1); 	\draw[fill] (11,0) circle (.1); 	\draw[fill] (9,3.8) circle (.1); \draw[fill] (9,1.27) circle (.1);
				\draw (7,-0.5) node {$1$}; 	
				\draw (11,-0.5) node {$3$}; 	\draw (9, 4.2) node {$2$}; 
				\draw (8.7, 1.5) node {$4$};
				
				\draw (-2.5,-2) node {$\Des(p)=$}; 	
				\draw (2,-2) node {$\{1,3\}$}; 
				\draw (9,-2) node {$\varnothing$};	
				\end{tikzpicture}
			\end{center}
			\caption{Descent sets of Gallai partitions}\label{fig:Des1}
		\end{figure}
	\end{example}
	
	
	Denote  the set of Gallai $k$-partitions of $K_n$ by $G_{n,k}$ and the set of transitive $k$-partitions of $\overrightarrow K_n$ by $T_{n,k}$. We prove the following.
	
	\begin{theorem}\label{m1g}
		For every $n>k\ge 1$, 
		the quasisymmetric functions 
		\[	
		\Q(G_{n,k}):=\sum\limits_{p\in G_{n,k}} \F_{\Des(p)}
		\]
		and	
		\[
		\Q(T_{n,k}):=\sum\limits_{p\in T_{n,k}} \F_{\Des(p)}
		\]
		are symmetric and Schur-positive. 
	\end{theorem}
	
	
	
	
	For maximal transitive and Gallai partitions we have the following explicit descriptions. 
	Here $\ch$ is the Frobenius characteristic map from the ring of class functions on symmetric groups to the ring of symmetric functions; for a definition see Section~\ref{sec:Schur2}.  
	
	\begin{theorem}\label{m3d}
		For every $n> 1$, 
		\[
		\Q(T_{n,n-1}) 
		= \ch \left( \chi^{(n-1,n-1)} \downarrow_{\symm_n}^{\symm_{2n-2}} \right),
		\]
		where $\chi^{(n-1,n-1)}$ is the irreducible $\symm_{2n-2}$-character indexed by $(n-1,n-1)$.  
	\end{theorem}
	
	For the undirected case we have the following.
	
	\begin{theorem}\label{m3u}
		For every $n> 1$, 	
		\[
		\Q(G_{n,n-1}) 
		= \ch \left(\left( \sum_{r=0}^{n-1} a_r \chi^{(n-1+r,n-1-r)} \right) \downarrow_{\symm_n}^{\symm_{2n-2}} \right),
		\]
		where $a_{r}$ is the number of perfect matchings of $2r$ points on a line with no short chords.  
	\end{theorem}
	
	\begin{remark}
		For the numbers $a_r$, see~\cite[A000806]{Sloane},~\cite{Marmor} and references therein. 
	\end{remark}
	
	It follows that the distribution of singleton blocks of edges of the type $\{(i,i+1)\}$ on maximal transitive edge partitions of the transitive tournament $\overrightarrow K_n$ 
	is equal to the distribution of the (standard) descent set on indecomposable $321$-avoiding permutations in the symmetric group $\symm_n$  (Theorem~\ref{thm:321} below).

	\medskip
	
	The rest of the paper is organized as follows. Gallai and transitive colorings of matroids and oriented matroids are studied in Section~\ref{sec:matroids}. This includes a tight upper bound on the maximal number of colors, 
	polynomiality, and several 
	interpretations of transitive 2-colorings. 
	In Section~\ref{sec:maximal} we count maximal colorings of directed and undirected complete graphs. 
	In Section~\ref{sec:Schur} we equip transitive and Gallai partitions of complete graphs with a natural descent map, determined by singleton blocks.    
	The resulting quasisymmetric functions 
	are shown to be symmetric and Schur-positive. 
	Section~\ref{sec:final} concludes the paper with a brief discussion of 
	maximal transitive partitions of Coxeter root systems and related algebras.

	\section{Gallai and transitive colorings of matroids}\label{sec:matroids}
	
	Let $M$ be a matroid on a ground set $E$ and let $C$ be a set of colors.
	Let $\EEE_M(C)$ be the set of Gallai colorings of $M$ with colors from $C$.   
	Consider first what happens for very short circuits.
	
	\begin{remark}\label{rem:must_be_loopless_intro}
		\begin{itemize}
			\item[(a)] If $\{e\}$ is a circuit of size $1$, then $e$ is called a {\em loop}. By Definition~\ref{def:Gallai}, in a Gallai coloring a loop can use only $0$ colors! This means that $\EEE_M(C) = \varnothing$ unless $M$ is loopless. We shall therefore always assume that the matroid $M$ is loopless. 
			
			\item[(b)] If $\{e_1, e_2\}$ is a circuit of size $2$, then the elements $e_1$ and $e_2$ are called {\em parallel}. In a Gallai 
			coloring of $M$, parallel elements always have the same color.    
		\end{itemize}
	\end{remark}
	
	\subsection{Maximal number of colors}
	
	By Remark~\ref{rem:must_be_loopless_intro}.1, a matroid has at least one Gallai coloring if and only if it is loopless.
	
	\begin{defn}\label{d.Gallai_max_num_colors}
		For a loopless matroid $M$, let $g(M)$
		be the maximal number of colors in a Gallai coloring of $M$.
	\end{defn}
	
	Clearly, $g(M) = 0$ if and only if $E = \varnothing$.
	
	\begin{theorem}\label{t.Gallai_max_eq_rank}
		For any loopless matroid $M$,
		\[
		g(M) = \rank(M). 
		\]
	\end{theorem}
	
	For the proof we need the following three properties of matroids. The first two are standard; see, e.g., \cite{Oxley}. The third property is a strengthening of a standard fact, and we therefore provide a proof.
	
	\begin{lemma}\label{t:unique_circuit}
		Let $B$ be a basis in a matroid.
		For each $e \in E \setminus B$ there exists a unique circuit containing $e$ and contained in $B \cup \{e\}$.    
	\end{lemma}
	
	\begin{lemma}\label{t:rank_and_circuit}
		For any $S \subseteq E$ and $e \in S$, 
		$\rank(S) = \rank(S \setminus \{e\})$ if and only if $e$ belongs to a circuit contained in $S$.    
	\end{lemma}
	
	\begin{lemma}\label{t:circuit_intersection} {\rm (Strong elimination property)}
		Let $X_1$ and $X_2$ be two distinct circuits in a matroid, let $e \in X_1 \setminus X_2$ and let $f \in X_1 \cap X_2$. 
		Then $(X_1 \cup X_2) \setminus \{f\}$ contains a circuit which contains $e$.  
	\end{lemma}
	
	\begin{proof}[Proof of Lemma~\ref{t:circuit_intersection}]
		Consider the ranks of the sets 
		$X_1 \cup X_2$,
		$(X_1 \cup X_2) \setminus \{e\}$,
		$(X_1 \cup X_2) \setminus \{f\}$, and
		$(X_1 \cup X_2) \setminus \{e,f\}$.
		The element $e$ belongs to a circuit ($X_1$) contained in $X_1 \cup X_2$. Therefore, by Lemma~\ref{t:rank_and_circuit},
		\[
		\rank(X_1 \cup X_2) = \rank((X_1 \cup X_2) \setminus \{e\}).
		\]
		Similarly, the element $f$ belongs to a circuit ($X_1$ or $X_2$) contained in $X_1 \cup X_2$, and therefore
		\[
		\rank(X_1 \cup X_2) =
		\rank((X_1 \cup X_2) \setminus \{f\}).
		\]
		Finally, the circuit $X_2$ contains $f$ and is contained in $(X_1 \cup X_2) \setminus \{e\}$, so that
		\[
		\rank((X_1 \cup X_2) \setminus \{e\}) = \rank((X_1 \cup X_2) \setminus \{e,f\}).
		\]
		It follows that all four sets have the same rank, and in particular
		\[
		\rank((X_1 \cup X_2) \setminus \{f\}) = \rank((X_1 \cup X_2) \setminus \{e,f\}).
		\]
		By Lemma~\ref{t:rank_and_circuit}, this implies that $e$ belongs to a circuit contained in $(X_1 \cup X_2) \setminus \{f\}$, as claimed.
	\end{proof}
	
	\begin{proof}[Proof of Theorem~\ref{t.Gallai_max_eq_rank}]
		Let $M$ be a loopless matroid on a set $E$.
		We can assume that $\rank(M) \ge 1$, otherwise necessarily $E = \varnothing$, since $M$ is loopless, in which case clearly $g(M) = 0 = \rank(M)$. 
		By Observation~\ref{t:must_be_loopless_intro}, $g(M) \ge 1$.
		
		Let $C$ be a set (of colors) of size $g(M)$, and let $\ve: E \to C$ be a surjective Gallai coloring of $M$. 
		Pick a set $S \subseteq E$ of size $g(M)$ which is ``$\ve$-rainbow'', namely, all its elements are assigned distinct colors by $\ve$.
		Every subset of $S$ is clearly also $\ve$-rainbow and therefore, by Definition~\ref{def:Gallai}, $S$ does not contain a circuit. It follows that $S$ is an independent set in $M$, and in particular $g(M) = |S| \le \rank(M)$.
		
		In order to prove the opposite inequality $g(M) \ge \rank(M)$, we now construct a Gallai coloring of $M$ using exactly $r := \rank(M)$ colors.
		Indeed, let $B \subseteq E$ be a basis of $M$; 
		of course, $|B| = r$.
		Map $B$ bijectively onto the set of colors $C := [r]$. 
		By Lemma~\ref{t:unique_circuit}, for any $e \in E \setminus B$ there exists a unique circuit $\{e\} \subseteq X_e \subseteq B \cup \{e\}$.
		The set of colors is totally ordered: $ 1 < \ldots < r$; assign to $e$ the smallest color of an element of $X_e \setminus \{e\} \subseteq B$,
		noting that this set is nonempty since $M$ is loopless.
		The resulting coloring $\ve: E \to C$ is clearly surjective; we claim that it is also Gallai.
		
		We want to show that no circuit in $M$ is $\ve$-rainbow. Assume, on the contrary, that there is a rainbow circuit in $M$. 
		Let $X$ be a circuit with the following two properties;
		\begin{enumerate}
			\item 
			There is a unique element of $X$ having the smallest color (among the elements of $X$).
			\item 
			$|X \setminus B|$ is minimal, among the circuits having property (1).
		\end{enumerate}
		Note that property (1) is weaker than being rainbow, but it is exactly what we need for the forthcoming argument. In any case, the existence of a rainbow circuit implies the existence of $X$.
		
		Denote $k := |X \setminus B|$. Clearly $k \ne 0$, since the circuit $X$ is not independent. If $k = 1$ and $X \setminus B = \{e\}$, then $e \in E \setminus B$ and $X$ is the unique circuit containing $e$ and contained in $B \cup \{e\}$. By the definition of $\ve$, the color of $e$ is equal to the smallest color of an element of $X \setminus \{e\}$, contradicting property (1). Thus $k \ge 2$.
		
		By property (1), there is a unique element $e \in X$ having the smallest color. It may or may not belong to $B$, but since $k = |X \setminus B| \ge 2$ there is at least one other element $f \in X \setminus B$, and its color is not minimal in $X$. Let $X_f$ be the unique circuit containing $f$ and contained in $B \cup \{f\}$.
		
		Consider the circuits $X$ and $X_f$. Clearly $f \in X \cap X_f$ and $e \in X$. Also, $e \not\in X_f$ since all the elements of $X_f$ have colors larger or equal to the color of $f$, while $e$ has a strictly smaller color. Thus $e \in X \setminus X_f$, and by Lemma~\ref{t:circuit_intersection} there is a circuit $X'$ containing $e$ and contained in $(X \cup X_f) \setminus \{f\}$. This circuit has property (1), with $e$ as the unique element with smallest color; 
		and also $|X' \setminus B| < |X \setminus B|$, since $X_f \setminus \{f\} \subseteq B$ while $f \in X \setminus B$ and $f \not\in X' \setminus B$.
		This contradicts the choice of $X$ and shows that, indeed, no circuit is $\ve$-rainbow.
		Thus $\ve$ is a surjective Gallai coloring, completing the proof.
	\end{proof}
	
	\begin{corollary}\label{t.Gallai_max_graph}
		The maximal number of colors in a Gallai coloring of a graph $G$ on $n$ vertices with $c$ connected components is $n-c$.
	\end{corollary}
	
	\begin{proof}
		The rank of the graphic matroid corresponding to such a graph $G$, namely the number of edges in a spanning forest, is $n-c$.
	\end{proof}
	
	The proof 
	of Theorem~\ref{t.Gallai_max_eq_rank} 
	does not extend to oriented matroids. However, the corresponding statement does hold.
	
	Recall that, by Observation~\ref{t:must_be_acyclic_intro}, an oriented matroid has at least one transitive coloring if and only if it is acyclic.
	
	\begin{defn}\label{d.transitive_max_num_colors}
		For an acyclic oriented matroid $M$, let $t(M)$
		be the maximal number of colors in a transitive coloring of $M$.
	\end{defn}
	
	Clearly, $t(M) = 0$ if and only if $E = \varnothing$.
	
	\begin{theorem}\label{t:max_transitive_eq_rank}
		For any acyclic oriented matroid $M$,  
		\[
		t(M) = \rank(M).
		\]
	\end{theorem}
	
	For the proof, let us cite the following definition and basic results regarding orthogonality in an oriented matroid.
	
	\begin{defn}\label{d:orthogonality}{\rm \cite[inline definition before Proposition 3.4.1]{Bjorner}}
		Let $M$ be an oriented matroid on a set $E$.
		Two signed sets $X$, $Y$, with supports $\uX, \uY \subseteq E$, are said to be {\em orthogonal}, denoted by $X \perp Y$, if 
		either $\uX \cap \uY = \varnothing$ 
		or the restrictions of $X$ and $Y$ to their intersection are neither equal nor opposite, i.e., there exist $e, f \in \uX \cap \uY$ with signs satisfying $X(e)Y(e) = -X(f)Y(f)$.
	\end{defn}
	
	\begin{lemma}\label{t:orthogonality}{\rm \cite[Theorem 3.4.3]{Bjorner}}
		In an oriented matroid, if $X$ is a circuit and $Y$ is a cocircuit then $X \perp Y$.
	\end{lemma}
	
	\begin{lemma}\label{t:positive_cocircuit}{\rm \cite[Proposition 3.4.8]{Bjorner}}
		In an acyclic oriented matroid on a set $E$, every $e \in E$ is contained in a positive cocircuit.
	\end{lemma}
	
	\begin{proof}[Proof of Theorem~\ref{t:max_transitive_eq_rank}]
		Let $M$ be an acyclic oriented matroid on a set $E$.
		Let $C$ be a set (of colors) of size $t(M)$, and let $\ve: E \to C$ be a surjective transitive coloring of $M$. 
		Pick a set $S \subseteq E$ of size $t(M)$ which is ``$\ve$-rainbow'', namely, all its elements are assigned distinct colors by $\ve$.    
		Every subset of $S$ is clearly also $\ve$-rainbow and therefore, by Definition~\ref{def:transitive}, $S$ does not contain (the support of) a circuit. It follows that $S$ is an independent set in $M$, and in particular $t(M) = |S| \le \rank(M)$.
		
		In order to prove the opposite inequality $t(M) \ge \rank(M)$, we now construct a transitive coloring of $M$ using exactly $r := \rank(M)$ colors. The construction is recursive, depending on $r$.
		Of course, $r = 0$ is possible (for an acyclic, and in particular loopless, oriented matroid) only for $E = \varnothing$, and then indeed we use no colors. Assuming $r \ge 1$, Denote $E_r := E$. By Lemma~\ref{t:positive_cocircuit}, each element of $E_r$ is contained in a positive cocircuit. Choose a positive cocircuit $Y_r$ in $E_r$, and denote $E_{r-1} := E_r \setminus \uYr$. Since $Y_r$ is a cocircuit, the restriction of $M$ to $E_{r-1}$ is acyclic of rank $r-1$.
		Continue in this fashion to define subsets $E = E_r \supset E_{r-1} \supset \ldots \supset E_0 = \varnothing$ such that $E_i \setminus E_{i-1}$ supports a positive cocircuit $Y_i$ in $E_i$ for each $1 \le i \le r$. Finally, color the elements of $Y_i$ by color $i$, for each $1 \le i \le r$.
		The resulting coloring $\ve: E \to [r]$ is clearly surjective, and we claim that it is also transitive.
		
		Indeed, let $X$ be a signed circuit in $M$. There is a unique index $1 \le i \le r$ such that $\uX \subseteq E_i$ but $\uX \not\subseteq E_{i-1}$. It follows that $\uX \cap \uYi = \uX \cap (E_i \setminus E_{i-1}) \ne \varnothing$. By Lemma~\ref{t:orthogonality}, $X \perp Y_i$, and therefore the restrictions of $X$ and $Y_i$ to their (nonempty) intersection are neither equal nor opposite. Since $Y_i$ is positive, it follows that there are $e, f \in \uX \cap \uYi$ such that $e \in X^+$ and $f \in X^-$. By the definition of the coloring, $\ve(e) = \ve(f) = i$ and therefore $\ve(X^+) \cap \ve(X^-) \ne \varnothing$. Thus $\ve$ is transitive.
	\end{proof}
	
	
	
	\subsection{Polynomiality}
	
	Let $M$ be an acyclic oriented 
	matroid $M$ on a nonempty finite ground set $E$.
	For a positive integer $k$ 
	denote 
	\[
	\EEE_M(k):=\{\ve:E\longrightarrow [k]:\ \ve\ \text{{\rm is a transitive coloring}}\} 
	\]
	the set of transitive $k$-colorings of $M$. 
	Recall Definition~\ref{def:transitive_partition} of a {\em transitive $k$-partition}.
	
	\begin{proposition}\label{t.polynomial}\cite[Conjecture 3.5]{BGL}
		For any 
		acyclic oriented 
		matroid $M$ on a nonempty finite ground set $E$ 
		there exists a polynomial $p_M(x) \in x\ZZ[x]$ --- the {\em transitivity polynomial} of $M$ --- such that, for any positive integer $k$ 
		\[
		|\EEE_M(k)| = p_M(k).
		\]
		Moreover, there exist nonnegative integers $a_{M,j}$ such that
		\[
		p_M(x) = \sum_{j \ge 1} a_{M,j} (x)_j
		\]
		where  $(x)_j := x(x-1) \cdots (x-j+1)$ and $a_{M,j}$ is the number of transitive $j$-partitions of $M$.
	\end{proposition}
	
	\begin{proof}
		Each function $\ve: E \to [k]$ defines a partition $P_\ve$ of the set $E$, where $e_1, e_2 \in E$ belong to the same block if $\ve(e_1) = \ve(e_2)$.
		Let $\Pi(E)$ be the set of all partitions of $E$. 
		For each partition $P \in \Pi(E)$ and 
		positive integer $k$, define
		\[
		n_P(k) := |\{\ve \in \EEE_M(k) \,:\, P_\ve = P\}|.
		\]
		We shall prove that, for each partition $P$ of $E$ into $j$ nonempty parts, either
		\[
		n_P(k) = 0 \quad (\forall\, k)
		\]
		or
		\[
		n_P(k) = (k)_j \quad (\forall\, k).
		\]
		This will clearly complete the proof of the theorem, with the explicit expression 
		\begin{align*}
			a_{M,j} &= \text{number of partitions $P$ of $E$ into $j$ disjoint nonempty blocks such that } (\exists k) \, n_P(k) \ne 0\\
			&=\text{number of transitive $j$-partitions of $M$}.
		\end{align*}
		Indeed, the partition $P_\ve = P$ determines, for each $e_1, e_2 \in E$, whether or not $\ve(e_1) = \ve(e_2)$.
		Therefore it also determines, for each $(X^+,X^-) \in Circ(M)$, whether or not $\ve(X^+)\cap \ve(X^-)\ne \varnothing$. 
		It therefore determines whether or not $\ve \in \EEE_M(k)$.
		Thus, given a partition $P$ of $E$ into $j$ parts, if $n_P(k) \ne 0$ for some 
		positive integer $k$ 
		then there exists a function $\ve_j \in \EEE_M(k)$ with $P_{\ve_j} = P$, and consequently, for any 
		positive integer $j\le k$ 
		all functions $\ve : E \to C$ for some subset $C\subseteq [k]$ of order $j$,  
		with $P_\ve = P$, are in $\EEE_M(k)$; their number is clearly $(k)_j$. 
	\end{proof}
	
	
	\begin{example}\label{example:n!}
		Let $\overrightarrow K_n$ be the transitive tournament on $n$ vertices.
		The coefficients $a_k$ ($1 \le k \le n-1$) were computed in \cite[Section 3]{BGL} for $n \le 8$. 
		For $k=2$ and every $n\ge 2$ the following holds: 
		\[
		p_{\overrightarrow K_n}(2) = n!
		\]
		This follows from  a natural bijection between $\EEE_{\overrightarrow K_n}(2)$ and the symmetric group $\symm_n$: 
		for $\ve\in \EEE_M(2)$ let $\pi_\ve\in \symm_n$ the permutation which satisfies $\pi(i)>\pi(j)\Longleftrightarrow \ve(i,j)=1$ 
		for all $1\le i<j\le n$. To verify that this is a bijection, recall that a set of ordered pairs $J\subseteq \{(i,j):\ 1\le i<j\le n\}$ is an inversion set of a permutation in $\symm_n$ if and only if both $J$ and its complement are transitive, see e.g.~\cite{Grinberg}. This result will be generalized in Subsection~\ref{sec:2c}. 
	\end{example}
	


	Recall Definition~\ref{def:Gallai_partition} of a {\em Gallai $k$-partition}.
	
	\begin{proposition}\label{G.polynomial}		The number of Gallai $k$-colorings of a 
		a finite matroid $M$ is a polynomial in $k$. 
		Moreover, there exist nonnegative integers $b_{M,j}$ such that
		\[
		p_M(x) = \sum_{j \ge 1} b_{M,j} (x)_j
		\]
		where $(x)_j := x(x-1) \cdots (x-j+1)$ and $b_{M,j}$ is the number of Gallai $j$-partitions of $M$.
	\end{proposition}
	
	Proof is similar to the proof of Proposition~\ref{t.polynomial} and is omitted. 
	
	
	
	
	
	\subsection{2-colorings}\label{sec:2c}
	
	Let $M$ be a loopless matroid on a set $E$. 
	Two distinct elements $e, f \in E$ are called {\em parallel} if $\{e,f\}$ is a circuit in $M$. Being parallel (or equal) is an equivalence relation on $E$, and the equivalence classes are called {\em parallel classes}.
	In a vector matroid, parallel vectors are (nonzero) scalar multiples of each other.
	In a graphic matroid, two edges are parallel if they have the same endpoints.
	
	\begin{proposition}\label{prop:2_Gallai}
		The number of Gallai 2-colorings of a loopless matroid $M$ is $2^p$, where $p$ is the number of parallel classes of elements of $E$. 
	\end{proposition}
	
	\begin{proof}
		A 2-coloring of $M$ is Gallai if and only if the number of colors used to color each circuit is strictly smaller than the size of the circuit.
		Circuits of size $1$ do not exist, since the matroid is loopless.
		In circuits of size $2$, the two elements of the circuit are required to have the same color. This means that parallel elements of $E$ must have the same color.
		For circuits of size greater than $2$ there is no restriction on the coloring, since we have only two colors.
		Therefore a 2-coloring of $M$ is Gallai if and only if any two parallel elements of $E$ have the same color. 
	\end{proof}
	
	Recall that a graph is {\em simple} if it has no loops or parallel edges.
	
	\begin{corollary}\label{cor:2Gallai_simple graph}
		Any 2-coloring of a simple graph is Gallai. Hence, the number of Gallai 2-colorings of a simple graph with $e$ edges is $2^e$. 
	\end{corollary}
	
	Let $M$ be an oriented matroid on a set $E$. Define an equivalence relation on $E$ by: $e \sim f$ if either $e = f$ or $\{e,f\}$ is contained in a circuit of $M$. The equivalence classes of this relation are the {\em connected components} of $M$. Note that $\{e\}$ is a connected component of size $1$ if and only if $e$ is either a {\em loop} (forming a circuit of size $1$) or an {\em isthmus} (not contained in any circuit). 
	
	Let $M$ be an oriented matroid on a set $E$, and let $A \subseteq E$ be an arbitrary subset. For any signed set $X$ with support $\uX \subseteq E$, let ${}_{-A}X$ be the signed set obtained from $X$ by reversing the signs of all the elements of $A \cap \uX$. The set $\{{}_{-A}X \,:\, X \text{ is a circuit in } M\}$ is the set of circuits of an oriented matroid, denoted ${}_{-A} M$. A {\em reorientation} of $M$ is any of the oriented matroids ${}_{-A}M$, for $A \subseteq E$.
	
	
	\begin{theorem}\label{thm:acyclic_2_colorings}
		Let $M$ be an acyclic oriented matroid on a nonempty set $E$.
		The number of transitive 2-colorings of $M$ is equal to $2^c$ times the number of acyclic reorientations of $M$,
		where $c$ is the number of connected components of $M$. 
	\end{theorem}
	
	\begin{proof}
		There is an obvious bijection between subsets of $E$ and 2-colorings of $E$ with colors from $\{0,1\}$: for any $A \subseteq E$, the characteristic function $\chi_A: E \to \{0,1\}$, with $\chi_A(e) = 1 \iff e \in A$, is a 2-coloring of $E$.
		Similarly, there is a natural mapping from subsets $A \subseteq E$ to reorientations ${}_{-A}M$ of $M$, but it is not a bijection. 
		For example, ${}_{-E}M = {}_{-\varnothing}M = M$, since reversing the signs of all the elements in every circuit of an oriented matroid yields the same oriented matroid. Similarly, ${}_{-A}M = M$ if $A \subseteq E$ is a connected component of $M$, since the support of each circuit is contained in a unique connected component of $M$.
		
		We claim that, for any $A, B \subseteq E$:
		${}_{-A}M = {}_{-B}M$ if and only if, for each connected component $C$ of $M$, $C \cap A$ is either $C \cap B$ or $C \cap (E \setminus B)$. 
		Indeed,  ${}_{-A}M = {}_{-B}M$ if and only if, for every circuit $X$ of $M$, ${}_{-A}X$ is either ${}_{-B}X$ or ${}_{-B}(-X)$. 
		Denote 
		$\bar{A} := E \setminus A$,
		$\bar{B} := E \setminus B$,
		$S_1 := (A \cap B) \cup (\bar{A} \cap \bar{B})$ and
		$S_2 := (A \cap \bar{B}) \cup (\bar{A} \cap B)$.
		Thus $S_2$ is the symmetric difference of $A$ and $B$, and $S_1 = E \setminus S_2$. 
		Clearly ${}_{-A}X = {}_{-B}X$ if and only if $X \subseteq S_1$, and ${}_{-A}X = {}_{-B}(-X)$ if and only if $X \subseteq S_2$. 
		Thus, by the definition of connected components in an oriented matroid, ${}_{-A}X = {}_{-B}(\pm X)$ for all the circuits $X$ in $M$ if and only if each connected component $C$ of $M$ is contained in either $S_1$ or $S_2$. Finally, $C \subseteq S_1$ is equivalent to $C \cap A = C \cap B$, whereas $C \subseteq S_2$ is equivalent to $C \cap A = C \cap \bar{B}$.
		
		Denoting by $c$ the number of connected components of $M$, it follows that there is a $2^c:1$ map from subsets (or 2-colorings) of $E$ to reoerientations of $M$.
		
		Let us now consider {\em acyclic} reorientations. We claim that, for any $A \subseteq E$, the reorientation ${}_{-A}M$ is acyclic if and only if the 2-coloring $\chi_A$ is transitive. 
		Rephrased contrapositively, it suffices to show that, for any circuit $X$ in $M$, the circuit ${}_{-A}X$ in ${}_{-A}M$ is either positive or negative if and only if $\chi_A(X^+) \cap \chi_A(X^-) = \varnothing$.
		Note that $X^+ \ne \varnothing$ and $X^- \ne \varnothing$, since $M$ is originally acyclic.
		Indeed, the circuit ${}_{-A}X$ is positive if and only if $X^+ \subseteq E \setminus A$ while $X^- \subseteq A$, and this is equivalent to $\chi_A(X^+) = \{0\}$ and $\chi_A(X^-) = \{1\}$. Similarly, the circuit ${}_{-A}X$ is negative if and only if $\chi_A(X^+) = \{1\}$ and $\chi_A(X^-) = \{0\}$.
		
		We conclude that there is a $2^c:1$ map from transitive 2-colorings of $E$ to acyclic reoerientations of $M$. This completes the proof.
	\end{proof}
	
	%
	
	\begin{remark}\label{rem:acyclic_2_colorings}
		The factor $2^c$, which appears in the above result for abstract oriented matroids, disappears (as we shall soon see) when the matroid is represented (say, over $\RR$), and in particular when it corresponds to a directed graph. 
		This is because 
		multiplying by $-1$ all the vectors
		in a connected component of 
		a represented oriented matroid
		(or reversing the directions of all the edges in a connected component of a directed graph)
		yields a different set of vectors
		(and a different graph), 
		but with the same oriented matroid.  
	\end{remark}
	
	For the next result we need a fundamental fact about linear inequalities, which is a consequence of Farkas' lemma. It is sometimes called Gordan's lemma~\cite[\S 1.4]{Border}.
	
	
	\begin{lemma}\label{t:Gordan}
		Let $A$ be a matrix over $\RR$.
		Then exactly one of the following claims is true.
		\begin{enumerate}
			\item[(a)] 
			There exists a (column) vector $x$ such that $x^t A > 0$.
			\item[(b)]
			There exists a (column) vector $y$ such that $y \ge 0$, $y \ne 0$ and $A y = 0$.
		\end{enumerate}
		All vector-to-zero inequalities here $(>,\ge)$ are component-wise.
	\end{lemma}
	
	For acyclic oriented matroids represented over $\RR$ the following holds.
	
	\begin{theorem}\label{thm:Orlik1} 
		For any acyclic oriented matroid $M$ represented over $\RR$, 
		the number of transitive $2$-colorings of $M$ is equal to the number of chambers in the dual hyperplane arrangement. 
	\end{theorem}	
	
	\begin{proof} 
		Let $M$ be an acyclic oriented matroid, represented by a finite (multi)set of vectors $E\subset \RR^n$. 
		Note that $0 \not\in E$, since $M$ is acyclic and therefore loopless.
		View the elements of $\RR^n$ as column vectors.
		Given a 2-coloring $\ve: E \to \{-1,1\}$, 
		define 
		\[
		C_\ve
		:= \bigcap_{e \in E} \{x \in \RR^n: \sign (x^t e) = \ve(e)\} .
		\]
		This is a (possibly empty) open chamber in the dual hyperplane arrangement, and all chambers are of this form. We first want to show that $C_\ve$ is nonempty if and only if $\ve$ is transitive.
		
		Let $M_\ve$ be the reorientation of $M$ corresponding to $\ve$, as in the proof of Theorem~\ref{thm:acyclic_2_colorings}.
		Let $A_\ve$ be the matrix with columns $\ve(e)e$, for $e \in E$.
		Then, by definition,
		\[
		C_\ve = \{x \in \RR^n \,:\, x^t A_\ve > 0\} ,
		\]
		while a positive circuit in $M_\ve$ corresponds to a vector $y \ge 0$, $y \ne 0$ with inclusion-minimal support such that $A_\ve y = 0$.
		Therefore, by Lemma~\ref{t:Gordan}, $C_\ve \ne \varnothing$ if and only if $M_\ve$ has no positive circuit, i.e., $M_\ve$ is acyclic.
		By the proof of Theorem~\ref{thm:acyclic_2_colorings}, this happens if and only if $\ve$ is transitive.
		
		It follows that the map $\varphi$, from the set of all transitive 2-colorings of $M$ to the set of all (nonempty) open chambers in the dual hyperplane arrangement, defined by $\varphi(\ve) := C_\ve$, is surjective.
		It is also clearly injective, since $\ve$ can be recovered from $C_\ve$:
		for any $e \in E$, $\ve(e)$ is the sign of $x^t e$ for at least one, equivalently all, of the vectors $x \in C_\ve$. 
		This completes the proof.  
	\end{proof}
	
	\begin{remark}
		Theorem~\ref{thm:Orlik1} 
		is related to a well-known theorem of 
		Orlik and Terao~\cite{OT}; see a brief discussion in Section~\ref{sec:OT}.
	\end{remark}
	
	\begin{corollary}
		For an acyclic oriented matroid $M$ 
		in an $n$-dimensional real vector space, 
		the number of transitive $2$-colorings of $M$ satisfies 
		\[
		|\EEE_M(2)|=(-1)^n \chi_{\mathcal A(M)}(-1),
		\]
		where $\chi_{\mathcal A(M)}$ is the characteristic polynomial of the hyperplane arrangement dual to $M$. 
	\end{corollary}
	
	\begin{proof}
		Let $\mathcal A$  be an hyperplane arrangement in $\RR^n$. 
		By Zaslavsky's Theorem~\cite[Theorem 2.5]{Stanley_H},     
		the number of chambers in $\mathcal A$ is equal to 
		$(-1)^n \chi_{\mathcal A}(-1)$. Theorem~\ref{thm:Orlik1} 
		completes the proof. 
	\end{proof}
	
	\begin{corollary}\label{cor:2-Coxeter} 
		For any finite Coxeter group $W$, 
		the number of transitive $2$-colorings of the set $\Phi^+(W)$ of positive roots 
		is equal to $|W|$. 
	\end{corollary}
	
	\begin{proof}
		By Theorem~\ref{thm:Orlik1}, the number of transitive $2$-colorings of $\Phi^+(W)$ is equal to the number of chambers in 
		the reflection arrangement associated to $W$, which, in turn, is equal to the number of elements in $W$~\cite[p.~123]{Bjorner-Brenti}.
	\end{proof}
	
	
	
	
	\begin{corollary}\label{cor:acyclic_orientations}
		The number of transitive 2-colorings of an acyclic directed graph $\overrightarrow G$ is equal to the number of acyclic reorientations of $G$. 
	\end{corollary}
	
	\begin{proof}
		Combine Theorem~\ref{thm:acyclic_2_colorings} with Remark~\ref{rem:acyclic_2_colorings}.
	\end{proof}
	
	
	
	Using a well-known result of Stanley, this corollary may be reformulated as follows.
	
	\begin{corollary} 
		For any acyclic directed graph $\overrightarrow G$ of order $n$, 
		the number of transitive $2$-colorings of $\overrightarrow G$ is equal to $(-1)^n f_G(-1)$, where $f_G(x)$ is the chromatic polynomial of the underlying undirected graph $G$.
	\end{corollary}
	
	\begin{proof}
		By~\cite[Corollary 1.3]{Stanley_DM_73}, the number of acyclic orientations of an undirected graph $G$ is equal to $(-1)^n f_G(-1)$, where $f_G(x)$ is the chromatic polynomial of $G$.  Combining this with Corollary~\ref{cor:acyclic_orientations} completes the proof. 
	\end{proof}
	
	\begin{remark} 
		The number of transitive 2-colorings of an acyclic directed graph depends on the underlying undirected graph, but not on the orientation. 
		This is a unique phenomenon for $k=2$ colors. 
		For example, the number of maximal transitive colorings of an $n$-cycle with $0< m<n$ clockwise edges is $m(n-m)(n-1)!$,  
		thus depends on $m$.
	\end{remark} 
	

	\section{Enumeration of maximal partitions of complete graphs}\label{sec:maximal}

	\subsection{
		Gallai partitions of the complete graph}
	
	Let $K_n = (V,E)$ be the (undirected) {\em complete graph} on $n$ vertices. 
	Thus $V = [n] := \{1, \ldots, n\}$ 
	and $E = \{\{i,j\} \,:\, i,j \in V,\, i < j\}$.

	
	It is well-known that the graphic matroid of any graph is representable over any field, and that the rank of the graphic matroid of $K_n$ is $n-1$; the bases are exactly the spanning trees of $K_n$.
	Theorem~\ref{t.Gallai_max_eq_rank} thus implies the following result. 
	
	\begin{corollary}\label{cor:max_complete}
		For the complete graph $K_n$, the maximal number of colors in a Gallai coloring is
		\[
		g(K_n) = n-1.
		\]
	\end{corollary}
	
	Corollary~\ref{cor:max_complete} was proved in~\cite[Appendix]{ESS}; 
	see the discussion preceding~\cite[Theorem JL]{Gouge_etal}.

	\begin{definition}
		A Gallai coloring of $K_n$ is called {\em maximal} if it uses the maximal possible number of colors, namely $n-1$.
	\end{definition}
	
	Recall the notion of {\em Gallai partition} from Definition~\ref{def:Gallai_partition}. 
	In the special case of complete graphs, this was  introduced in~\cite{KST92}.
	Each Gallai coloring of $K_n$ gives rise to a partition of the edge set $E$ into nonempty color sets. 

	\begin{definition}
		A Gallai partition is {\em maximal} if it has the maximal possible number of blocks, namely $n-1$.    
	\end{definition}
	
	The main result of the current subsection is the following.
	
	\begin{theorem}\label{t.number_max_Gallai_partitions}
		The number of maximal Gallai partitions of $K_n$ $(n \ge 2)$ is equal to $(2n-3)!!$.
	\end{theorem}
	
	Theorem~\ref{t.number_max_Gallai_partitions} 
	will be given two distinct proofs, one using hamiltonian paths and the other using complete bipartite subgraphs. Both proofs consist of a sequence of lemmas, some of which are of independent interest.
	
	\begin{remark}
		Gouge et al.~\cite{Gouge_etal} count Gallai colorings of $K_n$ up to renaming the colors as well as the vertices. 
		Gallai partitions, as defined above, correspond to renaming only the colors. Renaming the vertices may result in a different partition. 
	\end{remark} 
	
	\begin{definition}
		Let $\ve$ be a maximal Gallai coloring of $K_n$.
		An {\em $\ve$-rainbow hamiltonian path} is a (directed) path of length $n-1$, visiting each vertex exactly once,
		whose edges are assigned $n-1$ different colors by $\ve$.
	\end{definition}
	
	\begin{lemma}\label{t.hamiltonian_path}
		Every maximal Gallai coloring of $K_n$ $(n \ge 2)$ has an $\ve$-rainbow hamiltonian path.
	\end{lemma}
	
	\begin{proof}
		Assume that the longest $\ve$-rainbow path $P \subseteq E$ is of length $k-1$, namely visits $k$ vertices. If $k = n$, then $P$ is an $\ve$-rainbow hamiltonian path, and we are done. 
		Assume that $k < n$.
		
		Using the assumption that the coloring $\ve$ is maximal, extend $P$ to an $\ve$-rainbow set $T \subseteq E$ of size $n-1$ by adding edges of the missing colors. The set $T$ 
		contains no cycle, 
		since it is $\ve$-rainbow and $\ve$ is Gallai. Having size $n-1$, it is therefore a spanning tree of $K_n$; in particular, it is connected. Therefore there exists an edge $e \in T \setminus P$ which has precisely one vertex in common with the set of vertices of $P$. Denote the other vertex of $e$ by $v$, and the vertices of $P$, in order, by $v_1, \ldots, v_k$, starting from one of the endpoints of the path $P$. 
		
		The edge $e$ connects $v$ with one of $v_1, \ldots, v_k$ and has a {\em new} color, namely a color different from those of the edges of $P$. 
		Let $1 \le i \le k$ be the smallest integer such that the edge $\{v,v_i\}$ has a new color. 
		If $i = 1$ then $v, v_1, \ldots, v_k$ is the sequence of vertices of an $\ve$-rainbow path, contradicting the maximality of $k$. 
		Otherwise $i \ge 2$, and the color of $\{v,v_{i-1}\}$ is not new. Looking at the triangle $\{v,v_{i-1}\}, \{v_{i-1},v_i\}, \{v, v_i\}$, the color of $\{v,v_i\}$, which is new, is necessarily different from the colors of $\{v,v_{i-1}\}$ and of $\{v_{i-1},v_i\}$. Since $\ve$ is Gallai, the colors of $\{v,v_{i-1}\}$ and of $\{v_{i-1},v_i\}$ must be equal, and therefore $v_1, \ldots, v_{i-1}, v, v_i, \ldots, v_k$ is the sequence of vertices of an $\ve$-rainbow path, again contradicting the maximality of $k$.
		This completes the proof.
	\end{proof}
	
	\begin{lemma}\label{t.limited_color}
		Let $\ve$ be a maximal Gallai coloring of $K_n$, and let $v_1, \ldots, v_n$ be the sequence of vertices of an $\ve$-rainbow hamiltonian path. 
		Denote $c_i := \ve(\{v_i,v_{i+1}\})$ $(1 \le i \le n-1)$. Then:
		\begin{itemize}
			\item[(a)]
			For any $1 \le i < j \le n$,
			\[
			\ve(\{v_i,v_j\}) \in \{c_i, \ldots, c_{j-1}\} .
			\]
			\item[(b)]
			If $\ve(\{v_1,v_n\}) = c_k$ then,
			for any $1 \le i \le k$ and $k+1 \le j \le n$,
			\[
			\ve(\{v_i,v_j\}) = c_k .
			\]
		\end{itemize}
	\end{lemma}
	
	\begin{proof}
		{(a)}
		Fix $ 1\le i < j \le n$.
		If $j-i = 1$ then, by definition, $\ve(\{v_i,v_{i+1}\}) = c_i$. 
		Assume that $j-i \ge 2$.
		Since $\ve$ is Gallai, it assigns the same color to at least two of the edges in the cycle $\{v_i,v_{i+1}\}$, $\ldots$, $\{v_{j-1},v_j\}$, $\{v_j,v_i\}$. 
		The edges $\{v_i,v_{i+1}\},$ $\ldots,$ $\{v_{j-1},v_j\}$ are assigned distinct colors, since they belong to an $\ve$-rainbow path.
		Therefore $\{v_j,v_i\}$ has the same color as one of the other edges, namely $\ve(\{v_i,v_j\}) \in \{c_i, \ldots, c_{j-1}\}$.
		
		{(b)}
		Fix $1 \le i \le k$ and $k+1 \le j \le n$, and denote $c := \ve(\{v_i,v_j\})$.
		Consider the cycle $\{v_1,v_2\}$, $\ldots$, $\{v_{i-1},v_i\}$, $\{v_i,v_j\}$, $\{v_j,v_{j+1}\}$, $\ldots$, $\{v_{n-1},v_n\}$, $\{v_n,v_1\}$. The colors assigned to the edges are, respectively, $c_1,\ldots,c_{i-1},c,c_j,\ldots,c_{n-1},c_k$.
		By (a) above, $c \in \{c_i,\ldots,c_{j-1}\}$.
		Since $\ve$ is Gallai, at least two of the edges in the cycle are assigned the same color. It follows that $c_k$ is equal to one of the other colors, but since $i \le k \le j-1$ the only option is $c_k = c$. Thus $\ve(\{v_i,v_j\}) = c_k$, as claimed.
	\end{proof}
	
	\begin{definition}
		Let $\ve: E \to C$ be a coloring of the edge set $E$ of $K_n$. 
		A color $c \in C$ is called {\em singleton} if there is a unique edge with that color.
	\end{definition}
	
	\begin{lemma}\label{t.singleton}
		Every maximal Gallai coloring of $K_n$ $(n \ge 2)$ has a singleton color.
	\end{lemma}
	
	\begin{proof}
		By induction on $n$.
		The claim clearly holds for $n = 2$.
		Let $n > 2$, and assume that the claim holds for $K_m$ for all $2 \le m < n$.
		Let $\ve$ be a maximal Gallai coloring of $K_n$, and let $v_1, \ldots, v_n$ be the sequence of vertices of an $\ve$-rainbow hamiltonian path, which exists according to Lemma~\ref{t.hamiltonian_path}. 
		Denote $c_i := \ve(\{v_i,v_{i+1}\})$ $(1 \le i \le n-1)$ and assume, following Lemma~\ref{t.limited_color}(a), that $v(\{v_1,v_n\}) = c_k$.
		
		By Lemma~\ref{t.limited_color}(b), 
		$\ve(\{v_i,v_j\}) = c_k$ for any $1 \le i \le k$ and $k+1 \le j \le n$. 
		By Lemma~\ref{t.limited_color}(a), 
		$\ve(\{v_i,v_j\})\in \{c_1,\ldots,c_{k-1}\}$ for any $1 \le i < j \le k$, and 
		$\ve(\{v_i,v_j\})\in \{c_{k+1},\ldots,c_{n-1}\}$ for any $k+1 \le i < j \le n$.
		Since $n > 2$, at least one of $k$ and $n-k$ is larger than $1$.
		If $k \ge 2$ then the restriction of $\ve$ to the complete graph $K_k$ on the vertices $v_1,\ldots,v_k$ is a Gallai coloring that uses exactly the colors $c_1,\ldots,c_{k-1}$, and is therefore maximal. Since $k \le n-1$, the induction hypothesis implies that this restriction has a singleton color. This color is not used outside $K_k$, and is therefore a singleton color of $\ve$, as required.
		Similarly, if $n-k \ge 2$ then the restriction of $\ve$ to the complete graph $K_{n-k}$ on the vertices $v_{k+1},\ldots,v_n$ is maximal Gallai, and has a singleton color which is also singleton for $\ve$ itself.
		This completes the proof.
	\end{proof}
	
	Lemma~\ref{t.limited_color}(b) also implies the following well-known result, which will be used in 
	the proof of Lemma~\ref{lem:extend_Gallai_partition}
	and in a bijective proof of Lemma~\ref{lem:Gallai-matchings}.
	
	\begin{lemma}\label{lem:bipartite}\cite[Corollary 2.5]{Gouge_etal} 
		Any maximal Gallai coloring of $K_n$ $(n \ge 2)$ has a unique color $c$ such that the edges colored by $c$ span a complete bipartite graph on $n$ vertices.
		This is the only color that ``touches'' every vertex of $K_n$. 
		On each of the two parts, the induced coloring is also maximal Gallai.
	\end{lemma}
	
	Surprisingly, the number of $\ve$-rainbow hamiltonian paths is independent of $\ve$.
	
	\begin{lemma}\label{t.number_rainbow_hamiltonian_paths}
		Every maximal Gallai coloring of $K_n$ $(n \ge 2)$ has exactly $2^{n-1}$ (directed) rainbow hamiltonian paths.
	\end{lemma}
	
	\begin{proof}
		By induction on $n$.
		The claim clearly holds for $n = 2$: $K_2$ has two directed hamiltonian paths.
		Let $n > 2$, and assume that the claim holds for $n-1$.
		Let $\ve$ be a maximal Gallai coloring of $K_n$,
		let  $v_1, \ldots, v_n$ be the sequence of vertices of an $\ve$-rainbow hamiltonian path, 
		and denote $c_i := \ve(\{v_i,v_{i+1}\})$ $(1 \le i \le n-1)$.
		Assume that $c_k$ is a singleton color of $\ve$; its existence is guaranteed by Lemma~\ref{t.singleton}.
		
		Let $1 \le i \le k-1$, and consider the triangle $\{v_i,v_k\}, \{v_k,v_{k+1}\}, \{v_i, v_{k+1}\}$. The color $\ve(\{v_k,v_{k+1}\})$ $= c_k$ is singleton, and is therefore distinct from $\ve(\{v_i,v_k\})$ and $\ve(\{v_i,v_{k+1}\})$. The coloring $\ve$ is Gallai, and therefore $\ve(\{v_i,v_k\}) = \ve(\{v_i,v_{k+1}\})$. We can similarly show that $\ve(\{v_k,v_j\}) = \ve(\{v_{k+1},v_j\})$ for any $k+2 \le j \le n$. If we contract the edge $\{v_k,v_{k+1}\}$ to a single vertex, the coloring $\ve$ therefore induces a well-defined coloring $\ve'$ of the resulting graph $K_{n-1}$. This is clearly a maximal Gallai coloring, which does not use the color $c_k$.
		By the induction hypothesis, $K_{n-1}$ has $2^{n-2}$ (directed) $\ve'$-rainbow hamiltonian paths.
		
		Every $\ve$-rainbow hamiltonian path in $K_n$ must contain the edge $\{v_k,v_{k+1}\}$, which has a singleton color; it therefore restricts to a $\ve'$-rainbow hamiltonian path in $K_{n-1}$.
		Conversely, every $\ve'$-rainbow hamiltonian path can be extended to an $\ve$-rainbow hamiltonian path by blowing the vertex $v_k = v_{k+1}$ to an edge, and this can be done in exactly two ways: the vertex $v_k = v_{k+1}$ cuts the path in $K_{n-1}$ into two sub-paths, each having this vertex as an end-point, and there is a choice which of them to connect (in $K_n$) to $v_k$, while connecting the other one to $v_{k+1}$. 
		Note that at least one of the two sub-paths contains more than one vertex, since $n > 2$.
		It follows that there are $2^{n-1}$ $\ve$-rainbow hamiltonian paths in $K_n$, as claimed.
	\end{proof}
	
	We are now nearly at a position to give two proofs of Theorem~\ref{t.number_max_Gallai_partitions}.
	In fact, each proof will require only one additional lemma.
	
	\begin{lemma}\label{t.number_partitions_per_hamiltonian_path}
		The number of maximal Gallai partitions of $K_n$ $(n \ge 2)$ for which a given hamiltonian path is rainbow, is the Catalan number $C_{n-1} = \frac{1}{n} \binom{2n-2}{n-1}$.
	\end{lemma}
	
	\begin{proof}
		Let $v_1, \ldots, v_n$ be the sequence of vertices of some hamiltonian path in $K_n$, and fix a sequence of distinct colors $c_1,\ldots,c_{n-1}$. Maximal Gallai partitions for which this specific path is rainbow correspond bijectively to maximal Gallai colorings $\ve$ for which $\ve(\{v_i,v_{i+1}\}) = c_i$ $(1 \le i \le n-1)$.
		Let $a_n$ be the number of such colorings.
		
		Let $\ve$ be a maximal Gallai coloring of $K_n$ for which this hamiltonisn path is rainbow, and assume that $\ve(\{v_1,v_n\}) = c_k$, for some $1 \le k \le n-1$.
		By Lemma~\ref{t.limited_color}(b), $\ve(\{v_i,v_j\}) = c_k$ whenever $1 \le i \le k$ and $k+1 \le j \le n$.
		Also, by Lemma~\ref{t.limited_color}(a), the restriction of $\ve$ to the complete graph $K_k$ on the vertices $v_1, \ldots, v_k$ uses only the colors $c_1, \ldots, c_{k-1}$, and is maximal Gallai with the obvious rainbow hamiltonian path.
		Similarly for the restriction of $\ve$ to the complete graph $K_{n-k}$ on the vertices $v_{k+1}, \ldots, v_n$, using the colors $c_{k+1}, \ldots, c_{n-1}$. 
		It follows that 
		\[
		a_n = \sum_{k=1}^{n-1} a_k a_{n-k}
		\qquad (n \ge 2).
		\]
		This recurrence, together with the initial value $a_1 = 1$, show that $a_n = C_{n-1}$ for all $n \ge 2$, as claimed.
	\end{proof}
	
	\begin{remark}
		A result closely related to Lemma~\ref{t.number_partitions_per_hamiltonian_path} 
		is proved in~\cite[Corollary~2.7]{Gouge_etal}. 
	\end{remark}
	
	
	\begin{proof}[First proof of Theorem~\ref{t.number_max_Gallai_partitions}.]
		Let $p_n$ be the number of maximal Gallai partitions of $K_n$ $(n \ge 2)$. 
		Consider the pairs $(\pi, P)$, where $\pi$ is a maximal Gallai partition of $K_n$ and $P$ is a (directed) $\pi$-rainbow hamiltonian path.
		By Lemma~\ref{t.number_rainbow_hamiltonian_paths}, the number of such pairs is $2^{n-1} p_n$.
		On the other hand, the total number of (directed) hamiltonian paths in $K_n$ is $n!$. Therefore, by Lemma~\ref{t.number_partitions_per_hamiltonian_path}, the number of such pairs is $n! \cdot C_{n-1} = \frac{(2n-2)!}{(n-1)!}$.
		It follows that
		\[
		p_n 
		= \frac{(2n-2)!}{2^{n-1}(n-1)!} 
		= \frac{1 \cdot 2 \cdots (2n-3) \cdot (2n-2)}{2 \cdot 4 \cdots (2n-2)}
		= (2n-3)!! ,
		\]
		as claimed.
	\end{proof}

	
	\begin{lemma}\label{lem:extend_Gallai_partition}
		For $n \ge 2$, fix a complete subgraph $K_{n-1}$ of $K_n$. 
		Then there are exactly $2n-3$ maximal Gallai partitions of $K_n$ which extend any specified maximal Gallai partition of $K_{n-1}$.
	\end{lemma}
	
	\begin{proof}
		The proof is by induction on $n$.
		The claim is obvious for $n = 2$.
		Let $n > 2$, and assume that the claim holds for any $1 \le k \le n-1$.
		Fix a complete subgraph $K_{n-1}$ of $K_n$ and a maximal Gallai partition of $K_{n-1}$. Fixing colors for each of the $n-2$ blocks of the partition, as well as one additional color to be used outside $K_{n-1}$, defines a bijection between maximal Gallai partitions and maximal Gallai colorings, of both $K_n$ and $K_{n-1}$. For convenience we shall use, from now on, the language of colorings.
		
		Let $\ve$ be a maximal Gallai coloring of $K_{n-1}$.
		By Lemma~\ref{lem:bipartite}, there is a unique color $c$ such that the edges colored by $c$ in $\ve$ span a complete bipartite graph on $n$ vertices; call $c$ the {\em base color} of $\ve$.
		
		Let $\ve'$ be a maximal Gallai coloring of $K_n$ which extends $\ve$. It has all the old colors of $\ve$, plus one additional new color.
		We claim that the base color of $\ve'$ is either this new color, or the same as the base color of $\ve$. 
		Indeed, assume that the base color $c'$ of $\ve'$ is an old color which is not the base color $c$ of $\ve$.
		By Lemma~\ref{lem:bipartite}, since $c' \ne c$, there is at least one vertex $v$ of $K_{n-1}$ which the color $c'$ doesn't touch (in $K_{n-1}$). Also, since $c'$ is an old color, there is at least one old edge $e$ with this color. The two endpoints of $e$ are on distinct sides of the complete bipartite graph on $n$ vertices colored by $c'$; therefore one of them is not on the same side as $v$. The edge connecting this vertex to $v$ is therefore also colored $c'$, contradicting the choice of $v$. Therefore, indeed, $c'$ is either $c$, the base color of $K_{n-1}$, or the new color.
		
		If the base color $c'$ is the new color, then all the old vertices are on one side of the bipartite graph that $c'$ defines, and the new vertex constitutes the other side. It follows that all the new edges are colored $c'$. This indeed yields a (unique) maximal Gallai partition of $K_n$ extending the old one.
		
		On the other hand, if $c' = c$ then the new vertex of $K_n$ joins one of the two (nonempty) sides of the complete bipartite subgraph of the old $K_{n-1}$. Assume that the sizes of these sides are $k$ and $n-k-1$ $(1 \le k \le n-2)$. If the new vertex joins the side of size $k$, then all the edges connecting it to the other side are colored $c$. The (old) coloring of the complete subgraph $K_k$ is maximal Gallai, by Lemma~\ref{lem:bipartite}, and so is the (new) coloring of $K_{k+1}$ (on this side plus the new vertex). By the induction hypothesis, there are $2k-1$ ways to obtain such a new coloring of $K_{k+1}$. A similar argument holds if the new vertex belongs to the other side of $K_{n-1}$, yielding $2(n-k-1)-1$ extensions.
		
		It is easy to see that the above extended colorings yield distinct maximal Gallai partitions. Their number is 
		$1 + (2k-1) + (2(n-k-1)-1) = 2n-3$, as claimed.
	\end{proof}
	
	\begin{proof}[Second proof of Theorem~\ref{t.number_max_Gallai_partitions}]
		Since there is a unique (empty) maximal Gallai partition of $K_1$, the claim follows immediately from Lemma~\ref{lem:extend_Gallai_partition}, by induction on $n$. 
	\end{proof}

	\subsection{Transitive partitions of the tournament}
	
	Let ${\overrightarrow K_n} = (V,E)$ be 
	the {\em transitive tournament} on $n$ vertices. 
	This is a directed graph, with set of vertices $V = \{v_1, \ldots, v_n\}$ and set of directed edges $E = \{(v_i,v_j) \,:\, 1 \le i < j \le n\}$.
	This directed graph has no loops, and has a unique directed edge between any two distinct vertices, pointing from the vertex with smaller index to the vertex with a larger index.
	
	Recall 
	Theorem~\ref{t:max_transitive_eq_rank}. 
	The rank of the oriented graphic matroid of ${\overrightarrow K_n}$ is $n-1$, with (signed) bases corresponding to spanning trees of the underlying complete graph.
	
	\begin{corollary}\label{cor:tour_max}
		For the transitive tournament ${\overrightarrow K_n}$, the maximal number of colors in a transitive coloring is
		\[
		t({\overrightarrow K_n}) = n-1.
		\]
	\end{corollary}

	The following lemma is the transitive ananlogue of
	~\cite[Proposition 1.1]{Gouge_etal}. 
	
	\begin{lemma}\label{obs:trans}
		For a coloring $\ve $ of an acyclic directed graph, whose underlying graph is chordal, 
		the following are equivalent: 
		\begin{itemize}
			\item[(a)] 
			$\ve$ is transitive.
			\item[(b)] 
			For every triangle with directed edges 
			$(u,v)$, $(v,w)$ and $(u,w)$,
			\[
			\ve(u,w) \in 
			\{\ve(u,v),\ve(v,w)\} . 
			\]    
		\end{itemize}    
	\end{lemma}
	
	\begin{proof}
		Assume that (a) holds, namely that $\ve$ is transitive. Thus each cycle contains two edges of the same color but opposite orientations. In a triangle with directed edges $(u,v)$, $(v,w)$ and $(u,w)$, 
		the edge $(v,w)$ has orientation opposite to that of the other two. Therefore, by transitivity, the color of $(v,w)$ is equal to the color of (at least) one of the others.  
		This is exactly (b).
		
		In the other direction, assume that (b) holds. If $\ve$ is not transitive, there is a cycle with no two edges of the same color but opposite orientations. Consider such a cycle $c$, of minimal length $\ell$. Because of (b) and the assumption that the directed graph is acyclic, necessarily $\ell \ge 4$. The underlying undirected graph is assumed to be chordal; thus the cycle $c$ has a chord $e$. 
		Edges of $c$, together with $e$, form two cycles, $c_1$ and $c_2$, each of length at least $3$ but less than $\ell$. By minimality of $\ell$, in each of the two cycles there are two edges of the same color and opposite orientations. By the assumption on $c$, one of these edges must always be $e$.
		Let $e_1$ ($e_2$) be an edge in $c_1$ ($c_2$) with the same color as $e$ but opposite orientation. Then $e_1,e_2 \in c$ have the same color and opposite orientations, contradicting the choice of $c$. This proves that (a) holds.
	\end{proof}
	
	\begin{lemma}\label{lem:complete-Gallai-transitive}
		Maximal transitive colorings of $\overrightarrow K_n$ correspond bijectively to 
		the maximal Gallai colorings of $K_n$ for which the path $\{v_1,v_2\},\ldots,\{v_{n-1},v_n\}$ is rainbow.
	\end{lemma}
	
	\begin{proof}
		Clearly, every transitive coloring of a directed graph yields a Gallai coloring of the underlying undirected graph. Also, in a maximal transitive coloring $\ve$ of $\overrightarrow K_n$, the path $(v_1,v_2),\ldots,(v_{n-1},v_n)$ is rainbow. 
		Indeed, 
		in the cycle $(v_i,v_{i+1}),\ldots,(v_{j-1},v_j),(v_i,v_j)$, the edge $(v_i,v_j)$ has an opposite orientation to all other edges.  Hence, by transitivity,  
		\[
		\ve(v_i,v_j) \in 
		\{\ve(v_i,v_{i+1}),\ldots,\ve(v_{j-1},v_j)\}
		\subseteq \{\ve(v_1,v_2),\ldots,\ve(v_{n-1},v_n)\} 
		\qquad (\forall\, i<j).
		\]
		Maximality means that the number of colors used is $n-1$, hence $|\{\ve(v_1,v_2),\ldots,\ve(v_{n-1},v_n)\}|=n-1$. Thus the path $(v_1,v_2),\ldots,(v_{n-1},v_n)$ is rainbow. 
		
		To complete the proof we need to show that every Gallai coloring $\tilde\ve$ of $K_n$ for which the path $\{v_1,v_2\}$$,$\ldots$,$$\{v_{n-1},v_n\}$ is rainbow
		yields a transitive coloring of $\overrightarrow K_n$. 
		By Lemma~\ref{t.limited_color}(a), 
		\[
		\tilde\ve(v_i,v_j)\in\{\tilde\ve(v_i,v_{i+1}),\dots,\tilde\ve(v_{j-1},v_j)\} \qquad(\forall\, i<j).
		\]
		Hence for every $i<j<k$, $\tilde\ve(v_i,v_j)\ne \tilde\ve(v_j,v_k)$. It follows that
		\[
		\tilde\ve(v_i,v_k)\in \{\tilde\ve(v_i,v_j),\tilde\ve(v_j,v_k)\}\qquad (\forall\, i<j<k).
		\]
		By Lemma~\ref{obs:trans}, this condition implies transitivity. 
	\end{proof}
	
	The following theorem is the directed analogue of Lemma~\ref{t.number_partitions_per_hamiltonian_path}.
	
	\begin{theorem}\label{thm:Catalan}
		The number of maximal transitive partitions of ${\overrightarrow K_n}$ $(n\ge 2)$ is equal to the Catalan number $C_{n-1}:=\frac{1}{n}\binom{2n-2}{n-1}$.
	\end{theorem}
	
	\begin{proof}
		Combine Lemma~\ref{t.number_partitions_per_hamiltonian_path} with 
		Lemma~\ref{lem:complete-Gallai-transitive}. 
	\end{proof}
	


	%
	
	
	
	\begin{corollary}\label{cor:bipartite_trans}
		Let $n \ge 2$.
		\begin{itemize} 
			\item[(a)] 
			For every maximal transitive coloring of $\overrightarrow K_n$ there exists a unique color $c$, 
			for which the edges colored by $c$ span a complete bipartite graph on $n$ vertices.
			\item[(b)] 
			There exists $1\le k<n$, such that the sides of this bipartite graph are 
			$\{v_i:\ 1\le i\le k\}$ and $\{v_i:\ k<i\le n\}$.
		\end{itemize}
	\end{corollary}
	
	\begin{proof}
		{(a)}    
		Combine Lemma~\ref{lem:bipartite} with     Lemma~\ref{lem:complete-Gallai-transitive}. 
		
		{(b)}  
		Combine the proof of  Lemma~\ref{t.number_partitions_per_hamiltonian_path} with Lemma~\ref{lem:complete-Gallai-transitive}. 
	\end{proof}
	
	
	
	
	
	We further prove the following refinement.
	
	\medskip
	
	Denote by $T_{n,n-1}$  the set of all maximal transitive partitions of the transitive tournament ${\overrightarrow K_n}$.
	
	\begin{definition}
		\begin{itemize}
			\item[(a)] 
			Consider a maximal transitive partition of ${\overrightarrow K_n}$ $p\in T_{n,n-1}$. 
			A directed edge $(v_i,v_j)$, $i<j$, is a 
			{\em minimal} edge in $p$ if every edge $(v_a,v_b)$, $a<b$, in the block of 
			$(v_i,v_j)$ satisfies $i\le a$. 
			\item[(b)] 
			Let $\edgeright (p)$ be the number of minimal edges in $p$.
		\end{itemize}
	\end{definition}
	
	Carlitz and Riordan~\cite{CR} defined a $q$-Catalan number $C_n(q)$ using the recursion
	\[
	C_{n+1}(q)  
	:=  \sum\limits_{k=0}^n q^{(k+1)(n-k)} C_k(q) C_{n-k}(q)\qquad(n\ge 0)
	\]
	with $C_0(q) := 1$.
	
	\begin{proposition}\label{prop:q-analog}
		For every $n > 1$
		\[
		\sum\limits_{p\in T_{n,n-1}} q^{\edgeright(p)}
		= q^{\binom{n}{2}} C_{n-1}(q^{-1}).
		\]
	\end{proposition} 
	
	\begin{proof}
		Denote the number of non-minimal edges in a transitive partition $p$ by non-$\edgeright(p)$.
		It suffices to prove that for every $n\ge 2$
		\[
		\sum\limits_{p\in T_{n,n-1}} q^{\text{non}-\edgeright(p)}=C_{n-1}(q).
		\]
		The proof is by induction on $n$.
		For $n=2$ there are one edge in $\overrightarrow K_2$ and statement clearly holds.
		
		Assume that the statement is correct for all $k\le n$. 
		Consider a transitive partition of $\overrightarrow K_{n+1}$. 
		By Corollary~\ref{cor:bipartite_trans}, there exists a unique $1\le t<n$, such that 
		the edges in the block containing $(v_1,v_{n+1})$ 
		are $\{(v_i,v_j):\ i\le t <j\}$. 
		Thus there are $(t-1)(n+1-t)$ non-minimal edges in this block. All other blocks are either in the tournament $\overrightarrow K_t$ spanned by the first $t$ vertices, or in the tournament $\overrightarrow K_{n+1-t}$ spanned by the last $n+1-t$ vertices. 
		By the induction hypothesis, 
		\[
		\sum\limits_{p\in T_{n+1,n}} q^{\text{non}-\edgeright(p)}
		= \sum\limits_{t=1}^n q^{(t-1)(n+1-t)} C_{t-1}(q)C_{n-t}(q).
		\]
		Letting $k:=n-t$ the RHS is equal to 
		\[
		\sum\limits_{k=0}^{n-1} q^{(n-1-k)(k+1)} C_{n-1-k}(q)C_{k}(q)
		= C_{n}(q). 
		\]
	\end{proof}

	\section{Schur-positivity}\label{sec:Schur}
	
	Recall from Section~\ref{sec:Descents_Schur} 
	the definition of a descent set map on transitive and Gallai  partitions and the resulting quasisymmetric generating functions $\Q(T_{n,k})$ and $\Q(G_{n,k})$.  
	In this section we prove 
	Theorems~\ref{m1g},~\ref{m3d} and~\ref{m3u}. In Subsection~\ref{sec:Schur1}, it is shown that for every positive integers $n$ and $k$, both quasisymmetric functions $\Q(T_{n,k})$ and $\Q(G_{n,k})$ 
	are  symmetric and Schur-positive. In Subsection~\ref{sec:Schur2}, the symmetric group characters corresponding to 
	$\Q(T_{n,n-1})$ and $\Q(G_{n,n-1})$ are 
	explicitely described.  A bijection which relates maximal Gallai partitions to perfect matchings is 
	described in Subsection~\ref{sec:bijective}. In Subsection~\ref{sec:Schur4}, it is shown that the distribution of the descent set on 
	transitive colorings of the tournament $\overrightarrow K_n$ 
	is equal to its distribution on indecomposable 321-avoiding permutations in the symmetric group $\symm_n$.


	\subsection{Proof of Theorem~\ref{m1g}}\label{sec:Schur1}

	\begin{definition} A subset $J\subseteq [n-1]$ is {\em sparse} if it does 
		not contain any consecutive pair of elements. 
	\end{definition}
	
	\begin{observation}\label{obs:sparse1}
		For every transitive (or Gallai) partition $p$ of the complete directed (or undirected) graph, $\Des(p)$ is sparse.
	\end{observation}
	
	\begin{proof}
		If $\Des(p)$ is not sparse then there exists an $i$ such that $i,i+1 \in \Des(p)$. By the definition of the descent set, it follows that the edges $(i,i+1)$ and $(i+1,i+2)$ form singleton blocks. Thus $(i,i+1)$, $(i,i+2)$ 
		and $(i,i+2)$ 
		belong to three different blocks, 
		and therefore form a rainbow triangle.
		This contradicts the assumption of $p$ being a transitive (or Gallai) partition.
	\end{proof}
	
	Denote 
	$g(n,k) := |G_{n,k}|$ 
	and $t(n,k) := |T_{n,k}|$.
	
	\begin{lemma}\label{lem:Des class cardinality}
		For every $n>k\ge 1$ and a sparse subset 
		$\varnothing \ne J\subseteq [n-1]$ 
		\[
		|\{p\in G_{n,k} \,:\, J \subseteq \Des(p) \}|
		= g(n-|J|,k-|J|) 
		\]
		and
		\[ 
		|\{p\in T_{n,k} \,:\, J \subseteq \Des(p) \}|
		= t(n-|J|,k-|J|). 
		\]
	\end{lemma}	
	
	\begin{proof}
		We prove the lemma for Gallai partitions.  
		The proof for transitive partitions is similar.  
		
		Let $p$ be a Gallai partition. 
		For every $i\in \Des(p)$, $(i,i+1)$ is a singleton block. 
		Since $p$ is a Gallai partition, it contains no rainbow triangle. Hence, for every $j\ne i, i+1$ the edges $(i,j)$ and $(i+1,j)$ belong to the same block. 
		It follows that the set of Gallai $k$-partitions of $K_n$ with a descent at $i$ is in bijection with the set of Gallai $(k-1)$-partitions of $K_n/(i,i+1)$ (edge contraction), which is isomorphic to $K_{n-1}$. 
		This proves the lemma for $|J| = 1$. Proceed by induction on the size of $J$.
	\end{proof}
	
	The following is a weak version of a new criterion 
	of Marmor 
	for Schur-positivity.  
	
	\begin{lemma}\cite[Theorem 1.8]{Marmor}\label{lem:Marmor1}
		Let $A$ be a set equipped with a descent set map, and assume that for every $a \in A$, $\Des(a)$ is sparse. 
		If for every sparse $J \subseteq [n-1]$, the cardinality of the set $\{a \in A \,:\, J \subseteq \Des(a)\}$ depends only on the size of $J$, then $\Q(A)$ is symmetric and Schur-positive. 	
	\end{lemma}	
	
	\begin{proof}[Proof of Theorem~\ref{m1g}.]
		Combine Lemma~\ref{lem:Marmor1}  with Lemma~\ref{lem:Des class cardinality}. 
	\end{proof}

	\subsection{Proofs of Theorems~\ref{m3d} and~\ref{m3u}}\label{sec:Schur2}
	
	
	We begin with some necessary background. 
	A {\em partition} of a positive integer $n$ is a weakly decreasing sequence $\lambda=(\lambda_1,\ldots,\lambda_t)$ of positive integers whose sum is $n$. We denote $\lambda \vdash n$.
	For a partition $\lambda\vdash n$ let 
	$\SYT(\lambda)$ be the set
	of standard Young tableaux of shape $\lambda$. We use the
	English convention, according to which row indices increase from 
	top to bottom. See~\cite[p. 312]{EC2} for definition and examples.
	
	Recall the descent set of a standard Young tableau $T$ of size $n$
	\[
	\Des(T):=\{i\in [n-1]:\ i+1\ \text{ appears in a lower row than }\ i\}.
	\]
	
	Let $s_\lambda$ be the Schur function indexed by the partition $\lambda$. 
	The following key theorem is due to Gessel.
	
	\begin{theorem}{\rm \cite[Theorem 7.19.7]{EC2}}\label{G1} 
		For every integer parition $\lambda \vdash n$,
		\[
		\Q({\SYT(\lambda)})=s_{\lambda}. 
		\]
	\end{theorem}
	
	
	There is a dictionary relating symmetric functions to 
	class functions on 
	the symmetric group. 
	The irreducible characters of $\symm_n$ are indexed by partitions $\lambda \vdash n$ and denoted $\chi^\lambda$. 
	The {\em Frobenius characteristic map} $\ch$ from class functions on $\symm_n$ to symmetric functions is defined by $\ch(\chi^{\lambda}) = s_{\lambda}$, and extended by linearity.
	Theorem~\ref{G1} may then be restated as follows:
	\[
	\ch(\chi^\lambda) 
	= \sum_{T \in SYT(\lambda)} \F_{n,\Des(T)}.
	\]
	A combinatorial rule for the restriction of irreducible $\symm_n$-characters was given by Young~\cite[Theorem 9.2]{James}:
	
	\begin{theorem}\label{BranchingRule} (The Branching Rule)
		For $\lambda\vdash n$
		\[
		\chi^\lambda \downarrow_{\symm_{n-1}}^{\symm_n} 
		= \sum_{\substack{\mu \vdash n-1 \\ |\lambda/\mu| = 1}} \chi^\mu.
		\]
	\end{theorem}
	
	Viewing tableaux of shape $\mu$ as tableaux of shape $\lambda$ with the entry $n$ ``forgotten'', the Branching Rule may be restated as
	\[
	\ch(\chi^\lambda \downarrow_{\symm_{n-1}}^{\symm_n}) = \sum_{T \in SYT(\lambda)} \F_{n-1,\Des(T) \cap [n-2]}.
	\]
	Iteration immediately gives the following.
	
	\begin{corollary}\label{quasi_restriction}
		For every $\lambda \vdash n$ and $m\le n$
		\[
		\ch(\chi^\lambda \downarrow_{\symm_m}^{\symm_n}) = \sum_{T \in \SYT(\lambda)} \F_{m, \Des(T) \cap [m-1]}.
		\]
	\end{corollary}
	
	The following lemma is folklore. 
	
	\begin{lemma}\label{lem:SYT222}
		Let $J \subseteq [2n-3]$ be a subset of size $m$. Then 
		\[
		|\{T \in \SYT((n-1,n-1)) \,:\, J \subseteq \Des(T)\}| 
		= C_{n-m-1}
		\]
		(the Catalan number) if 
		$J$ is sparse, and is zero otherwise.
	\end{lemma}
	
	\begin{proof}
		First, the descent set of a standard Young table of two row shape has no consecutive entries. Thus   
		\[
		|\{T \in \SYT((n-1,n-1)) \,:\, J \subseteq \Des(T)\}| 
		= 0
		\]
		if $J$ is not sparse.
		
		To prove the statement for sparse subsets, recall that a Dyck path of length $2n$ is a series of steps from $(0,0)$ to $(2n,0)$, starting at the origin $t_0:=(0,0)$, 
		where the $i$-th step is either 
		$t_i:=(1,1)+t_{i-1}$ (upper step) or $t_i:=(1,-1)+t_{i-1}$ (lower step),  
		provided that $t_i$ is not below the $x$-axis. 
		Denote by ${\mathcal D}_{2n}$ the set of Dyck paths of length $2n$, and by 
		\[
		{\rm Peak}(d)
		:= \{i \,:\, t_i \text{ is an upper step and } t_{i+1} \text{ is a lower step}\}
		\]
		the set of peaks of $d\in {\mathcal D}_{2n}$.
		
		Recall the 
		bijection from $\SYT(n-1,n-1)$ to Dyck path from $(0,0)$ to $(2n-2,0)$, determined as follows: 
		the $i$-th step is upper 
		if $i$ is in the first row of $T$ and 
		lower if $i$ is in the second row. 
		Let $J \subseteq [2n-3]$ be a subset of order $m$. Assume that $J$ is a subset the peak set of a given Dyck path. 
		Deleting the $j$-th and $j+1$-st steps for every $j\in J$ yields a Dyck path of length $2n-2-2m$, while  
		re-adding these steps recovers the original Dyck path. 
		It follows that the set of Dyck paths of length $2n-2$ with peak set containing $J$ 
		is in bijection with Dyck paths of length $2n-2-2m$, whose number is $C_{n-m-1}$. We conclude that 
		\[
		|\{T\in \SYT(n-1,n-1) \,:\, J \subseteq \Des(T)\}|
		= |\{d \in {\mathcal D}_{2n-2} \,:\, J \subseteq {\rm Peak}(d)\}|
		= |{\mathcal D}_{2n-2-2m}|
		= C_{n-m-1}.
		\]
	\end{proof}
	
	
	\begin{lemma}\label{lem1111}
		Let $J \subseteq [n-1]$ be a subset of size $m$. Then 
		\[
		|\{p \in T_{n,n-1} \,:\, J \subseteq \Des(p)\}| 
		= C_{n-m-1}
		\]
		if 
		$J$ is sparse, and is zero otherwise.
	\end{lemma}
	
	\begin{proof}
		Combining Lemma~\ref{lem:Des class cardinality} with Theorem~\ref{thm:Catalan} we obtain that if $J$ is sparse then 
		\[
		|\{p \in T_{n,n-1} \,:\, J \subseteq \Des(p)\}|
		= t_{n-m,n-m-1}
		= C_{n-m-1}.
		\]
		If $J$ is not sparse then, by Observation~\ref{obs:sparse1}, there are no transitive partitions of $\overrightarrow K_n$ with descent set $J$, completing the proof.   
	\end{proof}
	
	
	\begin{proof}[Proof of Theorem~\ref{m3d}]
		Combining Lemma~\ref{lem1111} with Lemma~\ref{lem:SYT222} one obtains
		\[
		|\{p \in T_{n,n-1} \,:\, J \subseteq \Des(p)\}| 
		= |\{T \in \SYT(n-1,n-1) \,:\, J \subseteq \Des(T)\}| 
		\qquad (\forall\, J \subseteq [n-1]). 
		\]
		Hence
		\[
		\sum\limits_{p \in T_{n,n-1}} {\bf x}^{\Des(p)}
		= \sum\limits_{T \in \SYT(n-1,n-1)} {\bf x}^{\Des(T) \cap [n-1]} ,
		\]
		where 
		${\bf x}^J := \prod_{i \in J} x_i$. 
		Equivalently
		\[
		\Q(T_{n,n-1})
		= \sum\limits_{p \in T_{n,n-1}} \F_{\Des(p)}
		= \sum\limits_{T \in \SYT(n-1,n-1)} \F_{\Des(T) \cap [n-1]}.
		\]
		By 
		Theorem~\ref{G1} 
		together with Corollary~\ref{quasi_restriction}, the RHS is equal to $\ch (\chi^{n-1,n-1} \downarrow^{\symm_{2n-2}}_{\symm_n})$, completing the proof. 
	\end{proof}
	
	A similar proof implies 
	the undirected analogue.
	
	\begin{lemma}\label{lem11111}
		Let $J \subseteq [n-1]$ be a subset of size $k$. Then 
		\[
		|\{p \in G_{n,n-1} \,:\, J \subseteq \Des(p)\}|
		= (2n-2k-3)!!
		\]
		if 
		$J$ is sparse, and is zero otherwise.
	\end{lemma}
	
	\begin{proof} 
		The proof is the same as the proof of Lemma~\ref{lem1111}, with Theorem~\ref{thm:Catalan} replaced by Theorem~\ref{t.number_max_Gallai_partitions}.
	\end{proof}
	
	Denote by $M_{2n}$ the set of perfect matchings
	of $2n$ points on a line, labeled by $1,\dots,2n$. 
	For $m \in M_{2n}$ define the {\em Short match} set 
	\[
	\Short(m):=\{i:\ (i,i+1) \in m\}.
	\]
	
	\begin{observation}\label{obs:match} 
		For every subset $J \subseteq [2n-1]$ of size $k$,
		\[
		|\{m \in M_{2n} \,:\, J \subseteq \Short(m)\}|
		= (2n-2k-1)!!
		\]
		if $J$ is sparse, and is zero otherwise.
	\end{observation}
	
	We deduce the following.
	
	\begin{lemma}\label{lem:Gallai-matchings}
		For every $n \ge 2$
		\[
		\Q(G_{n,n-1})
		= \sum\limits_{p\in G_{n,n-1}} \F_{n,\Des(p)}
		= \sum\limits_{m\in M_{2n-2}} \F_{n,\Short(m) \cap [n-1]}.
		\]    
	\end{lemma}
	
	\begin{proof}
		Comparing Observation~\ref{obs:match} with Lemma~\ref{lem11111}, 
		one deduces that the descent set distribution on $G_{n,n-1}$ is equal to the distribution of short matches in $[n-1]$ on perfect matching in $M_{2n-2}$, implying the claim of the lemma. 
	\end{proof}
	
	The following theorem is due to Marmor.
	
	\begin{theorem}\label{thm:marmor2}\cite[Theorem 1.6]{Marmor}
		The set $M_{2n}$ is symmetric and Schur-positive with respect to $\Short$.  Furthermore, its Schur expansion is given by the following formula:
		\[
		\sum\limits_{m\in M_{2n}} \F_{2n,\Short(f)}        
		= \sum_{k=0}^n |\{ m\in M_{2n-2k} \,:\, \Short(m) = \varnothing\}| s_{2n-k,k}.
		\]
	\end{theorem}
	
	\begin{proof}[Proof of Theorem~\ref{m3u}.]  
		Combining Lemma~\ref{lem:Gallai-matchings} with Theorem~\ref{thm:marmor2}, Theorem~\ref{G1} and Corollary~\ref{quasi_restriction}, 
		we obtain  
		\begin{align*}
			\ch(\Q(G_{n,n-1})) 
			&= \ch\left(\sum\limits_{m \in M_{2n-2}} \F_{n,\Short(f) \cap [n-1]}\right) \\
			&= \sum_{k=0}^{n-1} |\{ m \in M_{2n-2-2k} \,:\, \Short(m) = \varnothing\}| \, \chi^{2n-2-k,k} \downarrow^{\symm_{2n-2}}_{\symm_n}. 
		\end{align*}
	\end{proof}

	\subsection{A bijection from maximal Gallai partitions to perfect matchings}
	\label{sec:bijective}
	
	
	\begin{proof}[A bijective proof of Lemma~\ref{lem:Gallai-matchings}]
		We describe a bijection 
		\[
		\varphi: G_{n,n-1} \to M_{2n-2}
		\]
		from the set $G_{n,n-1}$ of maximal Gallai partitions of $K_n$ to the set $M_{2n-2}$ of perfect matchings of $2n-2$ points labeled by $1,\dots,2n-2$, under which
		\[
		\Des(p)
		= {\rm Short}(\varphi(p)) \cap [n-1]
		\qquad (\forall\, p \in G_{n,n-1}).
		\]
		
		A {\em  binary total partition tree of $[n]$} is a rooted complete  binary tree with $n$ leaves whose vertices are labeled by subsets of $[n]$, as follows:  
		the leaves are labeled by all distinct singletons, 
		and every internal vertex (father) is labeled by the disjoint union of the sets labeling its two sons. These trees are studied in~\cite[\S 5.2]{EC2}.  
		Denote the set of binary total partition trees of $[n]$ by ${\rm BTPT}_n$.
		
		Define a bijection 
		\[
		\psi: G_{n,n-1} \to {\rm BTPT}_n
		\]
		from maximal Gallai partitions  
		to binary total partition trees of $[n]$, as follows.
		By Lemma~\ref{lem:bipartite}, translated from the language of Gallai colorings to the language of Gallai partitions,
		in any maximal Gallai partition of $K_n$ $(n \ge 2)$ there is a unique block such that the edges in the block span a complete bipartite graph on $n$ vertices, and the induced partition on the edges in each side of this bipartite graph is also maximal Gallai.
		Label the root of the tree by the set $[n]$.
		Label the two sons of the root by the two sides of the bipartition of $[n]$ corresponding to the bipartite graph.
		Continue labeling the sons of any labeled father, by induction; see Figures~\ref{fig:O1} and~\ref{fig:O2}. 
		
		The map $\psi$ is a bijection, since the Gallai partition $p$ can be recovered from the labeling of the tree $\psi(p)$, as follows.  
		For every edge $e=(v_i,v_j)$ in $K_n$ there exists a unique  pair of brothers (two sons with common father), such that $i$ belong to one of the brothers and $j$ to the other.  This pair is called the  separating pair of $e$. 
		Two edges 
		belong to the same block  if and only if they have the same  separating pair of brothers.
		
		A bijection 
		\[
		\phi: {\rm BTPT}_n \to M_{2n-2}  
		\]
		from binary total partition trees of $[n]$ to perfect matchings of $2n-2$ points is described in \cite[Example 5.2.6]{EC2}:  
		Let $T \in {\rm BTPT}_n$. 
		First inductively relabel the inner vertices (that is, vertices which are not leaves) excluding the root as follows.  
		If labels $1,\dots,m$ have been used then label by $m+1$ the vertex $v$ satisfying the following condition: 
		among all unlabeled vertices with both sons labeled 
		the vertex $v$ has a son with minimal labeling.
		We get a complete binary tree $\hat T$ whose vertices (excluding the root) are labeled by $\{1,\dots,2n-2\}$ . The  matched pairs in the perfect matching $\phi(T)$ are the pairs of brothers in $\hat T$. For an example see Figure~\ref{fig:O3}.
		
		Finally, let 
		\[
		\varphi := \phi \circ \psi. 
		\]
		The map $\varphi$ is a bijection, since both $\psi$ and $\phi$ are. 
		Also, the pair $(i,i+1)$ 
		is a (short) match in $\varphi(p)$ if and only if $i$ and $i+1$ are brothers in $\widehat {\phi(p)}$. 
		For $1 \le i<n$ this happens if and only if $i$ and $i+1$ are leaves and brothers in $\phi(p)$.  Then the father of the leaves labeled by $i$ and $i+1$ is labeled by $\{i,i+1\}$ in $\phi(p)$. 
		This is equivalent to the edge $(i,i+1)$ being a singleton block in $p$, namely, $i \in \Des(p)$. 
	\end{proof}
	
	
	\begin{example} 
		See Figures~\ref{fig:O1},~\ref{fig:O2} and~\ref{fig:O3}.
		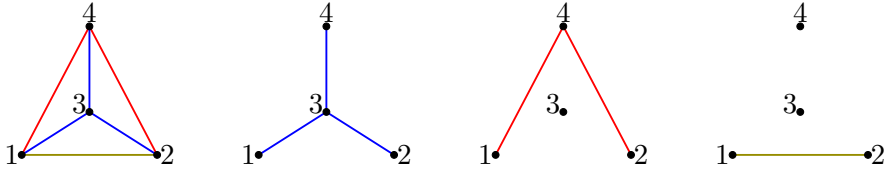
\begin{figure}[htb]
			\begin{center}
				\begin{tikzpicture}[scale=0.45]
				
				\draw[line width=0.25mm, red]			
				(-14,0)--(-12,3.8);
				\draw[line width=0.25mm, blue] (-10,0)--(-12,1.27)--(-14,0);
				\draw[line width=0.25mm, blue](-12,1.27)--(-12,3.8); 
				\draw[line width=0.25mm, red] (-12,3.8)--(-10,0); 
				\draw[line width=0.25mm, olive] (-10,0)--(-14,0);
				\draw (-12,4.2) node {$4$};
				\draw (-12.3,1.5) node {$3$};
				\draw (-9.7,0) node {$2$};
				\draw (-14.3,0) node {$1$};
				\draw[fill] (-14,0) circle (.1);
				\draw[fill] (-10,0) circle (.1);
				\draw[fill] (-12,3.8) circle (.1);
				\draw[fill] (-12,1.27) circle (.1);

				\draw[line width=0.25mm, blue] (-3,0)--(-5,1.27)--(-7,0);
				\draw[line width=0.25mm, blue](-5,1.27)--(-5,3.8); 
				\draw (-5,4.2) node {$4$};
				\draw (-5.3,1.5) node {$3$};
				\draw (-2.7,0) node {$2$};
				\draw (-7.3,0) node {$1$};
				\draw[fill] (-7,0) circle (.1);
				\draw[fill] (-3,0) circle (.1);
				\draw[fill] (-5,3.8) circle (.1);
				\draw[fill] (-5,1.27) circle (.1);

				\draw[line width=0.25mm, red]			(4,0)--(2,3.8);
				\draw[line width=0.25mm, red] (2,3.8)--(0,0); 
				\draw (2,4.2) node {$4$};
				\draw (1.7,1.5) node {$3$};
				\draw (-0.3,0) node {$1$};
				\draw (4.3,0) node {$2$};
				\draw[fill] (4,0) circle (.1);
				\draw[fill] (0,0) circle (.1);
				\draw[fill] (2,3.8) circle (.1);
				\draw[fill] (2,1.27) circle (.1);
				
				\draw[line width=0.25mm, olive] (7,0)--(11,0);
				\draw (9,4.2) node {$4$};
				\draw (8.7,1.5) node {$3$};
				\draw (6.7,0) node {$1$};
				\draw (11.3,0) node {$2$};
				\draw[fill] (11,0) circle (.1);
				\draw[fill] (7,0) circle (.1);
				\draw[fill] (9,3.8) circle (.1);
				\draw[fill] (9,1.27) circle (.1);									
				\end{tikzpicture}
			\end{center}
			\caption{A maximal Gallai coloring and its bipartitions}\label{fig:O1}
		\end{figure}
		
		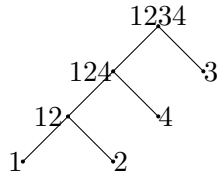
\begin{figure}[htb]
			\begin{center}
				\begin{tikzpicture}[scale=0.2]
				\draw[fill] (9,9) circle (.1); 	
				\draw (9,9.7) node {$1234$};
				\draw[fill] (6,6) circle (.1); 	
				\draw (4.5,6) node {$124$};
				\draw[fill] (12,6) circle (.1); 	
				\draw (12.5,6) node {$3$};
				\draw[fill] (3,3) circle (.1); 	
				\draw (1.7,3) node {$12$};
				\draw[fill] (9,3) circle (.1); 	
				\draw (9.5,3) node {$4$};
				\draw[fill] (0,0) circle (.1);
				\draw (-0.5,0) node {$1$};	 	
				\draw[fill] (6,0) circle (.1);
				\draw (6.5,0) node {$2$};									
				\draw (0,0)--(3,3)--(6,6)--(9,9)--(12,6);
				\draw (6,6)--(9,3);
				\draw (6,0)--(3,3);
				
				\end{tikzpicture}
			\end{center}
			\caption{The corresponding binary partition tree:}\label{fig:O2}
		\end{figure}
		
		\begin{figure}[htb]
			\begin{center}
				\begin{tikzpicture}[scale=0.2]
				\draw[fill] (9,9) circle (.1); 	
				\draw (9,9.7) node {};
				\draw[fill] (6,6) circle (.1); 	
				\draw (4.5,6) node {};
				\draw[fill] (12,6) circle (.1); 	
				\draw (12.5,6) node {$3$};
				\draw[fill] (3,3) circle (.1); 	
				\draw (1.7,3) node {};
				\draw[fill] (9,3) circle (.1); 	
				\draw (9.5,3) node {$4$};
				\draw[fill] (0,0) circle (.1);
				\draw (-0.5,0) node {$1$};	 	
				\draw[fill] (6,0) circle (.1);
				\draw (6.5,0) node {$2$};
				\draw (4.5,6) node {$6$};
				\draw (1.7,3) node {$5$};
				
				\draw (0,0)--(3,3)--(6,6)--(9,9)--(12,6);
				\draw (6,6)--(9,3);
				\draw (6,0)--(3,3);
				
				\end{tikzpicture}
			\end{center}
			\caption{The relabeld tree.\\
				The resulting perfect matching is $(1,2),(4,5),(3,6)$.}\label{fig:O3}
		\end{figure}
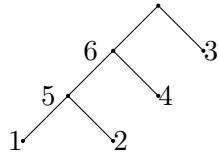  
	\end{example}

	
	
	
	
	

	\subsection{Indecomposable $321$-avoiding permutations}\label{sec:Schur4}
	A permutation $\pi$ in the symmetric group $\symm_n$ is {\em indecomposable} if 
	there is no $1 \le r < n$, 
	for which $\pi(i)<\pi(j)$ for all $i\le r< j$. 
	
	\begin{example} 
		The permutation $\pi = [31254] \in \symm_5$ is decomposable, since for 
		$r=3$, 
		$\pi(i) < \pi(j)$ for every $i \le 3<j$. 
		This may be viewed as a non-trivial principle block decomposition of the corresponding permutation matrix.  
		The permutation $\sigma = [43152]$ is indecompsable; indeed, there is no non-trivial principle block decomposition of its corresponding permutation matrix.	
		
		\[		
		\pi
		= \begin{pmatrix}
		0 & 1 & 0 & 0 & 0\\
		0 & 0 & 1 & 0 & 0\\
		1 & 0 & 0 & 0 & 0\\
		0 & 0 & 0 & 0 & 1\\
		0 & 0 & 0 & 1 & 0
		\end{pmatrix}\ \ , \ \  
		\sigma
		= \begin{pmatrix}
		0 & 0 & 1 & 0 & 0\\
		0 & 0 & 0 & 0 & 1\\
		0 & 1 & 0 & 0 & 0\\
		1 & 0 & 0 & 0 & 0\\
		0 & 0 & 0 & 1 & 0
		\end{pmatrix}
		\]		
	\end{example}
	

	
	
	Denote by {$\symm^*_n(321)$} the set of indecomposable permutations in $\symm_n$ with no decreasing subsequence of length 3.
	Recall the descent set of a permutation $\pi$ in the symmetric group $\symm_n$,
	\[
	\Des(\pi)
	:= \{i \,:\, \pi(i) > \pi(i+1)\}.
	\]
	
	\begin{theorem}\label{thm:321}
		For every $n \ge 2$,
		\[
		\sum\limits_{p \in T_{n,n-1}} {\bf x}^{\Des(p)} 
		= \sum\limits_{\pi \in \symm_n^*(321)}{\bf x}^{\Des(\pi)},
		\]
		where ${\bf x}^J := \prod\limits_{j \in J} x_j$.
		Equivalently, 
		\[
		\Q(T_{n,n-1}) = \Q(\symm_n^*(321)). 
		\]
	\end{theorem}
	
	\begin{proof}
		By~\cite[Theorem 1.2]{ABR},
		\[
		\Q(\symm_n^*(321))
		= \ch \left( \chi^{(n-1,n-1)} \downarrow_{\symm_n}^{\symm_{2n-2}} \right).
		\]
		Comparing this result with Theorem~\ref{m3d} gives
		\[
		\Q(T_{n,n-1})
		= \ch \left( \chi^{(n-1,n-1)} \downarrow_{\symm_n}^{\symm_{2n-2}} \right)=\Q(\symm_n^*(321)). 
		\]
	\end{proof}
	
	

	\section{Final remarks and open problems}\label{sec:final}
	
	\subsection{Maximal partitions of Coxeter root systems}
	
	Gallai and transitive colorings of abstract and vector matroids were discussed in Section~\ref{sec:matroids}. 
	Of special interest are the sets of positive roots of Coxeter systems.
	
	\begin{problem}
		Given a finite Coxeter group $W$, find the number of maximal transitive and Gallai partitions of the set $\Phi^+(W)$ of positive roots of $W$.
	\end{problem}
	
	For the dihedral group of order $2n$, $W=I_2(n)$,  the following holds.
	
	\begin{corollary} 
		For every integer $n > 1$:
		\begin{itemize}
			\item[(a)] 
			The number of maximal transitive partitions of $\Phi^+(I_2(n))$ is $n-1$.
			\item[(b)] 
			The number of maximal Gallai partitions of $\Phi^+(I_2(n))$ is $2^{n-1}-1$.    
		\end{itemize}
	\end{corollary}
	
	\begin{proof} 
		{(a)} 
		By Theorem~\ref{t:max_transitive_eq_rank}, a maximal transitive coloring of $\Phi^+(I_2(n))$ is a 2-coloring. 
		By Corollary~\ref{cor:2-Coxeter}, 
		the number of transitive 2-colorings of $\Phi^+(I_2(n))$ is 
		$|I_2(n)|=2n$. The number of transitive 2-partitions is obtained by ignoring the two monochromatic colorings and forgetting the names of the colors. Thus the number of maximal transitive partitions is equal to $(2n-2)/2 = n-1$.     
		
		{(b)} 
		The proof is similar to the proof of (a), 
		with Theorem~\ref{t:max_transitive_eq_rank} replaced by Theorem~\ref{t.Gallai_max_eq_rank}, and Corollary~\ref{cor:2-Coxeter} by Proposition~\ref{prop:2_Gallai}. 
	\end{proof}
	
	For the symmetric group $\symm_n$, namely 
	the Coxeter group of type $A_{n-1}$, 
	Theorems~\ref{m3d} and~\ref{m3u} may be reformulated as follows.
	
	\begin{theorem} 
		For every integer $n > 1$:
		\begin{itemize}
			\item[(a)] 
			The number of maximal transitive partitions of $\Phi^+(A_{n-1})$ is the Catalan number $C_{n-1}$.
			\item[(b)] 
			The number of maximal Gallai partitions of $\Phi^+(A_{n-1})$ is $(2n-3)!!$.  
		\end{itemize}
	\end{theorem}
	
	\begin{proof}
		Note that transitive (Gallai) partitions of the set of positive roots of type $A_{n-1}$,  
		\[
		\Phi^+ (A_{n-1})
		= \{e_i-e_j:\ 1\le i<j\le n\} ,
		\]
		may be interpreted as transitive (Gallai) partitions of the directed (undirected) complete graph on $n$ vertices. Theorems~\ref{m3d} and~\ref{m3u} complete the proof. 
	\end{proof}
	
	Regarding the Coxeter group of type $B_n$, we conjecture the following.
	
	\begin{conjecture}\label{conj:final1}
		The number of maximal transitive partitions of the set $\Phi^+(B_{n})$ of positive roots of type $B_n$ is
		\[
		C^B_n   
		:=  \sum_{k=0}^{n} \frac{3k+1}{n+k+1} \binom{2n-k}{n-2k}. 
		\] 
	\end{conjecture}
	Conjecture~\ref{conj:final1} was checked for $n \le 6$.  
	
	\begin{remark}
		The type $B$ Catalan number $C^B_{n-1}$ from Conjecture~\ref{conj:final1} is equal to the number of 
		ordered pairs $(\alpha,\beta)$ of compositions of $n$ with the same number of parts, such that
		$\sum\limits_{i=1}^{r} \alpha_i \ne \sum\limits_{i=1}^{r} \beta_i$ $(\forall\, r<\# \text{\rm parts)}$; see~\cite[A081696]{Sloane} and~\cite[Theorem 1.1]{Wilf}. 
		Note that the number of maximal transitive partitions of type $A_{n-1}$ is equal to the (type $A$) Catalan number $C_{n-1}$, 
		which counts pairs $(\alpha,\beta)$ of compositions of $n$ with the  same number of parts, such that 
		$\sum\limits_{i=1}^{r} \alpha_i \ge 
		\sum\limits_{i=1}^{r} \beta_i$ $(\forall\, r)$; see~\cite{Reifegerste, Stanley_Catalan}. 
	\end{remark}
	
	

	\subsection{Quasisymmetric functions}  
	
	Let $P_k(G)$ be the set of Gallai (respectively, transitive) $k$-partitions of the edge set of a loopless undirected (respectively, acyclic directed) graph $G$.
	
	\begin{definition} 
		The {\em descent set} of a Gallai (respectively, transitive) $k$-partition $p$
		of the edge set of a loopless graph (respectively, an acyclic directed graph) $G$  on the set of vertices $\{1, \ldots, n\}$ is 
		\[
		\Des(p) 
		:= \{ i \,:\, \text{the edge } (i,i+1) \text{ forms a singleton block in } p\}.
		\]
		An undirected (directed) graph $G$ is {\em $k$-Gallai} (respectively, {\em $k$-transitive}) {\em Schur-positive} if the quasisymmetric function 
		\[
		\Q(P_k(G)):=
		\sum\limits_{p\in P_k(G)} \F_{\Des(p)}  
		\]
		is symmetric and Schur-positive.
	\end{definition}
	
	By Theorem~\ref{m1g}, for any positive integers  $n>k\ge 1$, the complete graph $K_n$ is $k$-Gallai Schur-positive and the transitive tournament $\overrightarrow K_n$ is $k$-transitive Schur-positive. 
	
	Another family of Schur-positive graphs consist of cycles. 
	Let $\overrightarrow{\mathcal C}_n$ be the acyclic directed cycle with vertex set $[n]$ and edge set $\{(i,i+1) \,:\, 1 \le i < n\} \sqcup \{(1,n)\}$, and let ${\mathcal C}_n$ be the underlying undirected cycle.
	
	\begin{proposition}\label{cor:cycles} 
		For any $n \ge 3$ and $k \ge 1$, 
		the undirected cycle ${\mathcal C}_n$ is $k$-Gallai Schur-positive, 
		and the acyclic directed cycle $\overrightarrow{\mathcal C}_n$ is $k$-transitive Schur-positive.
	\end{proposition}
	
	\begin{proof}
		We will prove the result for the directed cycle 
		$\overrightarrow {\mathcal C}_n$. The proof for ${\mathcal C}_n$ is similar.
		
		First, notice that, for any fixed $i \in [n-1]$, the set of transitive $k$-partitions of $\overrightarrow{\mathcal C}_n$ with $i\in \Des(p)$
		is in bijection with the set of transitive $(k-1)$-partitions of 
		$\overrightarrow{\mathcal C}_n/(i,i+1)$ (edge contraction), which is isomorphic to $\overrightarrow{\mathcal C}_{n-1}$. This implies that for every $n \ge 3$ and $J \subseteq [n-1]$,
		\[
		|\{p \in P_k(\overrightarrow{\mathcal C}_n) \,:\, J \subseteq \Des(p)\}|
		= |P_{k-|J|}(\overrightarrow{\mathcal C}_{n-|J|})|.
		\]
		Note that in this case, the descent set of $p$ is not necessarily sparse. Moreover, all subsets of the same cardinality have the same fiber size. 
		This implies that there exist nonnegative integers 
		$\{m_{n,k,j}\}_{j=0}^{n-1}$, such that
		\[
		\Q(P_k(\overrightarrow{\mathcal C}_n))
		:= \sum_{p\in P_k(\overrightarrow{\mathcal C}_n)} \F_{\Des(p)}
		= \sum_{j=0}^{n-1} m_{n,k,j}
		\sum_{\substack{J \subseteq [n-1] \\ |J|=j}} \F_{J}
		= \sum_{j=0}^{n-1} m_{n,k,j} s_{(n-j,1^j)}.
		\]
		The last equality follows from Theorem~\ref{G1}. 
	\end{proof}
	
	In view of Theorem~\ref{m1g} and Proposition~\ref{cor:cycles}, we pose the following problem.
	
	\begin{problem}
		Characterize the $k$-Gallai Schur-positive loopless graphs and $k$-transitive Schur-positive acyclic directed graphs.
	\end{problem}
	
	A more general problem is the following.
	
	\begin{problem}
		Let $M$ be a finite loopless matroid (or acyclic oriented matroid), and let $P_k(M)$ be the set of its Gallai (respectively, transitive) $k$-colorings. 
		For which 
		set-valued functions $\Des: P_k(M) \to 2^{[\rank(M)]}$ is the quasisymmetric function 
		\[
		\sum_{p \in P_k(M)} \F_{\Des(p)}  
		\]
		symmetric and Schur-positive?
	\end{problem}

	\subsection{Algebras and Hilbert series}
	
	
	In this subsection we introduce two families of algebras,  
	intimately related to transitive and Gallai colorings. 
	
	\begin{definition}\label{def:transitive_algebra}
		Let $k$ be a positive integer, and let $M$ be an oriented matroid on a finite set $E$, with set of signed circuits $\Gamma(M)$.  
		The {\em transitive $k$-algebra} of $M$, denoted ${\mathcal T}_{M,k}$, is the commutative algebra over $\CC$ generated by $\{x_e:\ e\in E\}$ subject to the relations  
		\[
		\prod_{e_1 \in X^+, e_2 \in X^-} (x_{e_1} -x_{e_2}) = 0 
		\qquad (\forall\, (X^+,X^-) \in \Gamma(M))
		\]
		and 
		\[
		x_e^k = 1 
		\qquad (\forall\, e \in E).
		\]
	\end{definition}
	
	\begin{definition}\label{def:Gallai_algebra}
		Let $k$ be a positive integer, and let $M$ be a matroid on a finite set $E$, with set of circuits $\Gamma(M)$.  
		Let $\, <\,$ be an arbitrary linear order on $E$. 
		The {\em Gallai $k$-algebra} of $M$, denoted ${\mathcal G}_{M,k}$, is the commutative algebra over $\CC$ generated by $\{x_e:\ e\in E\}$ subject to the relations:  
		\[
		\prod_{\substack{e_1,e_2 \in X \\ e_1 < e_2}} (x_{e_1} - x_{e_2}) = 0 
		\qquad (\forall\, X \in \Gamma(M))
		\]
		and 
		\[
		x_e^k=1 \qquad (\forall\, e\in E).
		\]
	\end{definition}
	
	\begin{theorem}\label{t.dimTG}
		Let $k$ be a positive integer. 
		Then, for any finite oriented matroid $M$,
		\[
		\dim {\mathcal T}_{M,k}
		= \# \{\text{transitive\ $k$-colorings of}\ M\}
		\]
		and, for any finite matroid $M$,
		\[
		\dim {\mathcal G}_{M,k}
		= \# \{\text{Gallai\ $k$-colorings of}\ M\}.
		\]
	\end{theorem}
	
	\begin{proof} 
		Consider the set $V_k$ of all families $\{x_e\}_{e\in E}$ of points in $\CC^E$ which satisfy the defining relations of ${\mathcal T}_{M,k}$. 
		Clearly, $V_k$ is finite and Zariski closed in $\CC^E$. 
		Therefore ${\mathcal T}_{M,k}\cong \CC[V_k]$, the algebra of complex-valued functions on $V_k$.
		We claim that $V_k$ is in bijection with the set of transitive $k$-colorings of $M$.
		Indeed, let $\zeta$ be a fixed primitive complex $k$-th root of unity.  
		For each transitive $k$-coloring $\ve$ of $M$, define $x_e := \zeta^{\ve(e)}$ $(\forall e \in E)$. It is easy to verify that $\{x_e\}_{e \in E} \in V_k$, and that the mapping (from transitive $k$-colorings to elements of $V_k$) is a bijection.
		Since $\dim({\mathcal T}_{M,k}) = \dim \CC[V_k] = |V_k|$, 
		we conclude that 
		\[
		\text{dim}({\mathcal T}_{M,k})
		= \# \{\text{transitive\ $k$-colorings of}\ M\}
		\qquad (\forall k \ge 1).
		\]
		The proof for the Gallai algebra ${\mathcal G}_{M,k}$ is similar.
	\end{proof}
	
	Of special interest are the Gallai and transitive algebras of the set of positive roots of type $A_{n-1}$.
	
	\begin{definition}\label{def:T_and_G}	
		Let $n$ and $k$ be positive integers.
		\begin{itemize}
			\item[(a)]	
			The {\em transitive algebra} ${\mathcal T}_{n,k}:={\mathcal T}_{\overrightarrow K_n,k}$ 
			is the commutative algebra over $\CC$ generated  by $\{x_{ij}: 1\le i<j\le n\}$  subject to the relations
			\begin{align*}
				& (x_{im}-x_{ij})(x_{im}-x_{jm}) = 0 \qquad(\forall\, i<j<m),\\
				& x_{ij}^{k}=1 \qquad (\forall\, i<j).
			\end{align*}
			
			\item[(b)]	
			The {\em Gallai algebra} ${\mathcal G}_{n,k}:={\mathcal G}_{K_n,k}$ 
			is the commutative algebra over $\CC$ generated  by $\{x_{ij}: 1\le i<j\le n\}$  subject to the relations
			\begin{align*}
				& (x_{ij}-x_{im})(x_{ij}-x_{jm})(x_{im}-x_{jm}) = 0 \qquad(\forall\, i<j<m),\\
				& x_{ij}^{k}=1 \qquad (\forall\, i<j).
			\end{align*}
		\end{itemize}
	\end{definition}	
	
	\begin{corollary}\label{t.dimTG2}
		For all $n> 1$,
		\[
		{\rm{dim}}\, {\mathcal T}_{n,2}= n! ,
		\qquad
		{\rm{dim}}\, {\mathcal G}_{n,2}= 2^{\binom{n}{2}} ,
		\]
		\[
		{\rm{dim}}\, {\mathcal T}_{n,n-1}= C_{n-1},
		\qquad \text{\rm and } \qquad
		{\rm{dim}}\, {\mathcal G}_{n,n-1}= (2n-3)!! .
		\]
	\end{corollary}
	
	\begin{proof}
		All the claims follow from Theorem~\ref{t.dimTG} combined with specific enumeration results: 
		Corollary~\ref{cor:2-Coxeter} 
		(see also Example~\ref{example:n!}),  
		Corollary~\ref{cor:2Gallai_simple graph}, 
		Theorem~\ref{thm:Catalan1}, and Theorem~\ref{t.number_max_Gallai_partitions1}.
	\end{proof}
	
	
	
	Recall now the {\em Hilbert series} of a finitely generated 
	algebra ${\mathcal B}$ 
	\[    
	\Hilb({\mathcal B},q)
	:= \sum\limits_{k\ge 0} (\dim ({\mathcal B}_{\le j}) - \dim ({\mathcal B}_{\le j-1})) q^j\ .
	\]
	Here ${\mathcal B}_{\le j}$ is the degree $j$ filtered component of ${\mathcal B}$, where the filtered degree of each generator is 1.  
	
	
	
	\begin{conjecture}
		Let $[k]_j := \prod_{i=0}^{j-1} \frac{q^{k-i}-1}{q-1}$ $(k,j \ge 1)$.
		\begin{itemize}
			\item[(a)]  
			For all $n > 1$ and $k \ge 1$,
			\[
			\Hilb({\mathcal T}_{n,k},q)= 
			\sum\limits_{j=1}^{n-1} P_{n,j}(q) \cdot [k]_j, 
			\] 
			where $P_{n,1}(q),\ldots,P_{n,n-1}(q) \in \ZZ_{\ge 0}[q]$. 
			The leading coefficient satisfies $P_{n,n-1}(q) = q^{\binom{n-1}2} C_{n-1}$, where $C_{n-1}$ is the Catalan number. 
			
			\item[(b)] 
			For all $n > 1$ and $k \ge 1$, 
			\[
			\Hilb({\mathcal G}_{n,k},q)= 
			\sum\limits_{j=1}^{n-1} Q_{n,j}(q) \cdot [k]_j, 
			\] 
			where $Q_{n,1}(q),\ldots,Q_{n,n-1}(q) \in \ZZ_{\ge 0}[q]$.
		\end{itemize}
	\end{conjecture}
	
	Part (a) was checked for $n \le 8$. 
	Part (b) was checked for $n \le 5$. 
	
	
	\begin{remark} For all $j$, 
		{\em $P_{n,j}(1)$} is equal to the number of {\em transitive  partitions} with $j$ blocks of the edge set of $\overrightarrow K_n$, and {\em $Q_{n,j}(1)$} is equal to the number of {\em Gallai partitions} with $j$ blocks of the edge set of $K_n$. 
	\end{remark}
	
	
	A {\em Stirling permutation} of order $n$ is a permutation of the multiset 
	$\{1,1,2,2,...,n,n\}$ such that, for all $m$, all entries between two copies of $m$ are larger than $m$. 
	The {\em second-order Eulerian number} $E(n,j)$ counts the number of Stirling permutations of order $n$ with $j$ descents, see~\cite{Stirling}.

	
	\begin{conjecture} For any $n > 1$, 
		\[
		Q_{n,n-1}(q)
		=q^{\binom{n}{2}-1}\sum\limits_{j=0}^{n-1} E(n-1,j) q^{-j}. 
		\]
	\end{conjecture}

	\subsection{Transitive 2-colorings and Orlik-Terao algebras}\label{sec:OT} 
	
	The Orlik-Terao algebra of an hyperplane arrangement was introduced in~\cite{OT}. 
	For the sake of simplicity, we will only discuss the case of the reflection hyperplane arrangement of type $A_{n-1}$. 
	
	The following definition is for a general simple directed graph. 
	
	\begin{definition}\label{def:OT}
		Let $G$ be a simple directed graph with vertex set $V =[n]$ and edge set 
		$E \subseteq \{(i,j) \,:\, 1 \le i < j \le n\}$.
		The {\em Orlik-Terao algebra} of $G$, denoted ${\mathcal OT}(G)$, 
		is the commutative algebra over $\CC$, generated 
		by $\{x_{i,j}=-x_{j,i}:\ (i,j)\in E\}$ subject to the following relations:
		\begin{enumerate}
			\item[(a)]
			For every directed cycle $(e_1,\dots,e_t)$ in $G$,
			\[
			\sum_{j=1}^t \prod_{k \ne j} x_{e_k} = 0.  
			\]
			\item[(b)]
			For every $e \in E$,
			\[
			x_e^2 = 0 . 
			\]
		\end{enumerate}
	\end{definition}
	
	\begin{theorem}\label{thm:OT}\cite{OT}
		The dimension of the Orlik-Terao algebra of 
		a simple directed graph $G$ 
		is equal to the number of chambers in 
		its dual hyperplane arrangement. 
	\end{theorem}
	
	
	Let ${\mathcal OT}(A_{n-1})$ be the Orlik-Terao algebra of the hyperplane arrangement of type $A_{n-1}$, 
	or equivalently, of the transitive tournament $\overrightarrow K_n$.
	Note that ${\mathcal OT}(A_{n-1})$ is a graded algebra (since all its defining relations are homogeneous), while ${\mathcal T}_{n,2}$ is only a filtered algebra.
	For any algebra $A$ with a generating set $S$, let $Gr(A)$ be the associated graded of $A$ with respect to the filtration defined by $S$ (where the filtered degree of any element of $S$ is $1$).
	
	\begin{theorem}\label{cor:OT_gr} 
		For every $n> 1$,
		\[
		{\mathcal OT}(A_{n-1}) \cong Gr({\mathcal T}_{n,2})
		\]
		as graded algebras.
	\end{theorem}
	
	For the proof we need the following lemma.
	
	\begin{lemma}\label{lemma:basis}
		Let $B_n \subset {\mathcal T}_{n,2}$  be the set of all square-free monomials in $x_{ij}$ not containing products of the form $x_{im}x_{jm}$ for $i < j < m$. 
		Then $B_n$ is a basis of ${\mathcal T}_{n,2}$, compatible with the natural filtration.
	\end{lemma}
	
	\begin{proof}
		By Definition~\ref{def:T_and_G}(a), the defining relations of ${\mathcal T}_{n,2}$ are
		\[
		(x_{im}-x_{ij})(x_{im}-x_{jm}) = 0 \qquad(i<j<m) 
		\qquad \text{and} \qquad
		x_{ij}^2 = 1 \qquad (i<j).
		\]
		Rewrite these relations as
		\begin{equation}\label{eq:Tn2}
			x_{im} x_{jm} 
			= x_{ij} x_{jm} - x_{ij} x_{im} + 1
			\qquad (i < j < m) 
			\qquad \text{and} \qquad
			x_{ij}^2 = 1
			\qquad (i < j).        
		\end{equation}
		The set of all monomials in the $x_{ij}$ is, obviously, a spanning set for ${\mathcal T}_{n,2}$.
		Define a weight function on monomials in the $x_{ij}$ by 
		\[
		w \left( \prod_{i,j} x_{ij}^{m_{ij}} \right)
		:= \sum_{i,j} m_{ij} \cdot (i + j).
		\]
		Clearly, the weight of (each monomial in) the RHS of each of the relations in~\eqref{eq:Tn2} is strictly smaller than the weight of the LHS.
		This leads to a (non-deterministic) straightening algorithm (see, e.g. \cite[Section 2.2]{CLM}), as follows:
		Replace an (arbitrary) occurence of the LHS in a monomial by the RHS, recursively. 
		Each step of this algorithm leads to a monomial, or a linear combination of three monomials, of strictly smaller weights than the original.
		Thus the algorithm terminates after a finite number of steps, yielding a linear combination of monomials in $B_n$.
		In addition to the weight, the {\em degree} of each of these monomials is less than or equal to that of the original monomial.
		This shows that every filtered component of ${\mathcal T}_{n,2}$ is spanned by a subset of $B_n$, namely, that $B_n$ is compatibile with the filtration.
		
		It remains to show that $B_n$ is linearly independent (and, in particular, that its apparently distinct elements are indeed distinct) in ${\mathcal T}_{n,2}$. Since it spans ${\mathcal T}_{n,2}$, it suffices to show that $|B_n| \le \dim {\mathcal T}_{n,2}$.
		Recall that, by Corollary~\ref{t.dimTG2},
		$\dim {\mathcal T}_{n,2} = n!$.
		We shall prove, by induction on $n$, that $|B_n| \le n!$.
		Indeed, $|B_2| = |\{1, x_{12}\}| \le 2$. 
		For $n > 2$, each monomial in $B_n$ is square-free and contains $x_{in}$ for at most one index $1 \le i < n$.
		Therefore $B_n \subseteq \{1, x_{1n}, \ldots, x_{n-1,n}\} \cdot B_{n-1}$, 
		thus $|B_n| \le n \cdot |B_{n-1}|$. 
		This completes the proof. 
	\end{proof}
	
	\begin{proof}[Proof of Theorem~\ref{cor:OT_gr}]
		By Definition~\ref{def:OT}, the Orlik-Terao algebra ${\mathcal OT}(A_{n-1})$ is generated by $\{x_{ij} = -x_{ji} \,:\, 1\le i < j \le n\}$,
		subject to the relations
		\begin{equation}\label{eq:OT}
			x_{ij}x_{jm} + x_{jm}x_{mi} + x_{mi}x_{ij} = 0
			\qquad (i < j < m)
			\qquad \text{and} \qquad
			x_{ij}^2 = 0 
			\qquad (i < j).
		\end{equation}
		Notice that, due to the relations 
		$x_{ij} = -x_{ji}$, the relations~\eqref{eq:OT} are equivalent to 
		\[
		(x_{im}-x_{ij})(x_{im}-x_{jm}) = 0
		\qquad (i < j < m)     
		\qquad \text{and} \qquad
		x_{ij}^2 = 0 
		\qquad (i < j).
		\]
		On the other hand, these relations, 
		with $x_{ij}^2 = 0$ replaced by $x_{ij}^2 = 1$, are defining for  ${\mathcal T}_{n,2}$. Therefore, the assignments $x_{ij} \mapsto Gr(x_{ij})$ define a homomorphism of graded algebras
		\[
		{\mathcal OT}(A_{n-1}) \to Gr({\mathcal T}_{n,2}) . 
		\]
		To show that this an isomorphism, apply Lemma~\ref{lemma:basis}. Indeed, since $B_n$ is a basis of ${\mathcal T}_{n,2}$ compatible with the natural filtration, it canonically descends to a basis $Gr(B_n)$ of $Gr({\mathcal T}_{n,2})$. In particular, each element of $Gr(B_n)$ is a monomial in $Gr(x_{ij})$, hence $Gr({\mathcal T}_{n,2})$ is generated by $Gr(x_{ij})$ and the above homomorphism is surjective. It is also injective because $\dim {\mathcal OT}(A_{n-1}) = \dim Gr({\mathcal T}_{n,2}) = n!$. 
	\end{proof}
	
	
	
	
	\begin{corollary}
		For every $n> 1$
		\[
		\Hilb({\mathcal OT}(A_{n-1}))=\Hilb({\mathcal T}_{n,2})
		= \sum_{k=0}^{n-1} s(n,n-k)\, q^{k},
		\]
		where $s(n,k)$ are the Stirling numbers of the first kind.
	\end{corollary}
	
	\begin{proof}
		The first equality follows from Theorem~\ref{cor:OT_gr}.  
		By Lemma~\ref{lemma:basis},
		\[
		\Hilb({\mathcal T}_{n,2})
		= \sum_{m \in B_n} q^{\deg(m)} 
		= \prod_{k=0}^{n-1} (1+kq)
		= \sum_{k=1}^{n} s(n,n-k)\, q^k, 
		\]
		as claimed.
	\end{proof}

	\begin{remark}     
		Theorem~\ref{cor:OT_gr} may be generalized to any acyclic directed graph whose underlying undirected graph is chordal. 
	\end{remark}
	
	Further connections to the Orlik-Terao algebra will be discussed elsewhere. 
	

	

	
\end{document}